\documentclass[10pt,leqno]{article}
\usepackage[cm]{fullpage}
\usepackage{enumitem}
\usepackage{amscd}
\usepackage{ amssymb }
\usepackage{amsmath}
\usepackage{ esint }
\usepackage{systeme,mathtools}
\usepackage{amsfonts}
\usepackage{amsthm}
\usepackage{tikz-cd}
\usepackage{stmaryrd}
\usepackage{lipsum}
\usepackage{appendix}
\usepackage{hyperref}
\usepackage{cleveref}
\DeclareMathOperator{\divg}{div}
\DeclareMathOperator{\im}{Im}
\DeclareMathOperator{\tr}{tr}
\DeclareMathOperator{\Aut}{Aut}
\DeclareMathOperator{\End}{End}
\DeclareMathOperator{\Diff}{Diff}
\DeclareMathOperator{\GDiff}{GDiff}
\DeclareMathOperator{\Isom}{Isom}
\DeclareMathOperator{\Rc}{Rc}

\begin{document}

\theoremstyle{definition}
\newtheorem{claim}{Claim}
\theoremstyle{plain}
\newtheorem{proposition}{Proposition}[section]
\newtheorem{theorem}[proposition]{Theorem}
\newtheorem{lemma}[proposition]{Lemma}
\newtheorem{corollary}[proposition]{Corollary}
\theoremstyle{definition}
\newtheorem{defn}[proposition]{Definition}
\theoremstyle{remark}
\newtheorem{remark}[proposition]{Remark}
\theoremstyle{definition}
\newtheorem{example}[proposition]{Example}
\theoremstyle{definition}
\newtheorem*{Motivation}{Motivation}

\newcommand{\Addresses}{{
  \footnotesize

  \begin{normalsize}{\textbf{Kuan-Hui Lee}}
  \end{normalsize}
  \par\nopagebreak
  \textsc{Department of Mathematics, University of California-Irvine, CA, U.S.A }\par\nopagebreak
  \textit{E-mail address}: \texttt{kuanhuil@uci.edu}

}}

\title{\textbf{Stability and moduli space of generalized Ricci solitons}}
\author{KUAN-HUI LEE}
\date{}
\maketitle
\Addresses
\begin{abstract}
 The generalized Einstein-Hilbert action is an extension of the classic scalar curvature energy and Perelman’s $\mathcal F$-functional which incorporates a closed three-form.  The critical points are known as generalized Ricci solitons, which arise naturally in mathematical physics, complex geometry, and generalized geometry.  Through a delicate analysis of the group of generalized gauge transformations, and implementing a novel connection, we give a simple formula for the second variation of this energy which generalizes the Lichnerowicz operator in the Einstein case.  As an application we show that all Bismut-flat manifolds are linearly stable critical points, and admit nontrivial deformations arising from Lie theory. Furthermore, this leads to extensions of classic results of Koiso \cite{Ko4,Ko1,Ko2,Ko5} and Podesta, Spiro, Kröncke \cite{K1,K3,K4,K2} to the moduli space of generalized Ricci solitons.  To finish we classify deformations of the Bismut-flat structure on $S^3$ and show that some are integrable while others are not.
\end{abstract}

\section{Introduction}

Let $M$ be a smooth manifold and fix a closed 3-form $H_0$. Given a smooth family of Riemannian metrics and 2-forms  $(g_t,b_t)$, we say that $(g_t,b_t)$ is a \emph{generalized Ricci flow} (we will abbreviate it as GRF later on) if 
\begin{align}
    \frac{\partial}{\partial t}g=-2\Rc+\frac{1}{2}H^2, \quad
   \frac{\partial}{\partial t}b=-d^*H \quad \text{ where $H=H_0+db$}. \label{1}
\end{align}
This parabolic flow was written in \cite{P}, \cite{J} and it can be viewed as the Ricci flow using the Bismut connections $\nabla^{\pm}=\nabla\pm\frac{1}{2}g^{-1}H$. The generalized Ricci flow arises naturally in complex geometry \cite{Streets2020, 8189878}, mathematical physics \cite{polchinski_1998}  and generalized geometry \cite{MG},\cite{STREETS2017506}. We define the generalized Einstein--Hilbert functional
\begin{align}
    \mathcal{F}\colon\nonumber&\quad \Gamma(S^2M)\times \Omega^2\times C^\infty (M)\to \mathbb{R}
    \\& \quad (g,b,f)\longmapsto \int_M (R-\frac{1}{12}|H_0+db|^2+|\nabla f|^2)e^{-f}dV_g \label{2} 
\end{align}
and  
\begin{align}
    \lambda(g,b)\coloneqq\inf\Big\{\mathcal{F}(g,b,f)\big|\kern0.2em f\in C^\infty(M),\,\int_M e^{-f}dV_g=1\Big\}. \label{3}
\end{align}
One can see that $\lambda(g,b)$ can be achieved by some $f$ uniquely, i.e., $\lambda(g,b)=\mathcal{F}(g,b,f)$ and $\lambda$ is the first eigenvalue of the Schrödinger operator $-4\triangle+R-\frac{1}{12}|H_0+db|^2$. In \cite{P}, it was shown that $\lambda$ is monotone increasing under the generalized Ricci flow and critical points of $\lambda$ are steady gradient generalized Ricci solitons.

We say that a pair $\mathcal{G}(g,b)$ is a \emph{steady generalized Ricci soliton} if there exists a smooth vector field $X$ such that
\begin{align}
    0=\Rc_g-\frac{1}{4}H^2+\frac{1}{2}L_Xg, \quad 0=d^*_gH+i_XH. \label{4}
\end{align}
where $H=H_0+db$. In this work, we only focus on the case when $X=\nabla f$ for some smooth function $f$ and we say that $\mathcal{G}$ is a steady gradient generalized Ricci soliton. In fact, the first variation of $\lambda$ implies that if $\mathcal{G}(g,b)$ is a compact steady gradient generalized Ricci soliton, then $(g,b)$ satisfies (\ref{4}) with $X=\nabla f$ and $f$ is the minimizer of $\lambda(g,b)$.

The first goal of this work is to understand the variational structure of $\lambda$.  In \cite{K,P} second variation formulas were derived which employ the Levi-Civita connection and are difficult to understand geometrically due to the presence of many torsion terms.  Here we provide a conceptually distinct formulation which is the foundation of the results to follow.  The first point is to address the invariance of $\lambda$ under the group of generalized gauge transformations, which is the semidirect product of the group of diffeomorphisms with the space of $B$-field transformations.  To address this we employ a slice theorem shown in \cite{Rubio_2019,K} to reduce to certain nontrivial deformations.  On this restricted space of deformations we were able to discern a subtle structure in the second variation which leads to many applications.  The key point is the introduction of a modified connection $\overline{\nabla}$ on the variational space $T^*M \otimes T^*M$ which employs both Bismut connections $\nabla^{\pm}$. In particular, $\overline{\nabla}$ is defined by
\begin{align*}
    (\overline{\nabla}_X\gamma)(Y,Z)=\nabla_X\Big(\gamma(Y,Z)\Big)-\gamma(\nabla^-_XY,Z)-\gamma(Y,\nabla^+_XZ), \quad \text{ $\gamma\in\otimes^2T^*M$ and $X,Y,Z\in TM$.}
\end{align*}
Using the connection $\overline{\nabla}$ we are lead to the following conceptually clear formulation of the second variation of $\lambda$, which forms the foundation of the results to follow.

\begin{theorem}\label{T1}
Given a compact steady gradient generalized Ricci soliton $\mathcal{G}(g,b)$ on a smooth manifold $M$. Suppose $\mathcal{G}(g_t,b_t)$ is a one-parameter family of generalized metrics such that 
\begin{align*}
    &\frac{\partial}{\partial t}\Big|_{t=0}(g_t-b_t) =\gamma=h-K, \quad  (g_0,b_0)=(g,b).
\end{align*}
Let $u$ be the unique solution of 
\begin{align*}
    \triangle_f u=\overline{\divg}_f\overline{\divg}_f\gamma,\quad \int_M ue^{-f}dV_g=0.
\end{align*}
where the definition of $\overline{\divg}_f$ is given in (\ref{div}) and (\ref{div0}). The second variation of $\lambda$ on  $\mathcal{G}(g,b)$ is given by 
\begin{align*}
     \nonumber\frac{d^2}{dt^2}\Big|_{t=0}\lambda=\int_M \Big\langle \gamma, \frac{1}{2}\overline{\triangle}_f\gamma+\mathring{R}^+(\gamma)+\frac{1}{2}\overline{\divg}_f^*\overline{\divg}_f\gamma+\frac{1}{2}(\nabla^+)^2u \Big\rangle e^{-f}dV_g,
\end{align*} 
where $\overline{\divg}_f^*$ is the formal adjoint of $\overline{\divg}_f$ with respect to (\ref{6}), $\overline{\triangle}_f$ 
 is defined in (\ref{triangle}), $\langle \mathring{R}^+(\gamma),\gamma \rangle=R^+_{iklj}\gamma_{ij}\gamma_{kl}$ and $R^+$ is the Bismut curvature given in \Cref{P3}.

\end{theorem}

Clearly, we have the following corollary.
\begin{corollary}\label{C1}
Every compact, Bismut-flat manifold $(M,\mathcal{G})$ is linearly stable. The kernel of the second variation on a compact, Bismut-flat manifold consists of non-trivial 2-tensors $\gamma$ which are parallel with respect to $\overline{\nabla}$.
\end{corollary}
Note that in \cite{K}, the author proved that linear stability and dynamical stability are equivalent so we also have the following corollary.
\begin{corollary}
   Every compact, Bismut-flat manifold $(M,\mathcal{G})$ is dynamically stable, i.e., for any neighborhood $\mathcal{U}$ of $\mathcal{G}$, there exists a smaller neighborhood $\mathcal{V}$ such that the generalized Ricci flow starting in $\mathcal{V}$ will stay in $\mathcal{U}$ for all $t\geq 0$ and converge to a critical point of $\lambda$.
\end{corollary}

In the second part of this work, we  study the moduli space of generalized Ricci solitons. In a series of papers of Koiso \cite{Ko4,Ko1,Ko2} and \cite{FA,K3,K4}, authors discuss the moduli space of Einstein metrics and Ricci solitons. We extend their work to more general setting. Define an operator
\begin{align*}
        \mathcal{R}: \quad &\mathcal{GM}\longrightarrow T^*M\otimes T^*M
        \\& \mathcal{G}(g,b)\longmapsto \Rc^{H,f}=\Rc-\frac{1}{4}H^2+\nabla^2f-\frac{1}{2}(d^*H+i_{\nabla f}H) 
    \end{align*}
where $f$ is the minimizer of $\lambda(g,b)$. So the space of steady gradient generalized Ricci solitons $\mathcal{GRS}$ can be viewed as 
\begin{align*}
    \mathcal{GRS}=\mathcal{R}^{-1}(0).
\end{align*}
Using the generalized slice theorem, we will say that \emph{the premoduli space} of steady gradient generalized Ricci solitons at $\mathcal{G}$ is the set
\begin{align*}
\mathcal{P}_{\mathcal{G}}=\mathcal{GRS}\cap \mathcal{S}^{f}_{\mathcal{G}},
\end{align*}
where $\mathcal{S}_{\mathcal{G}}^f$ is the generalized slice constructed by \Cref{T5}.
\begin{defn}\label{D1}
Let $\mathcal{G}$ be a steady gradient generalized Ricci soliton and let $\mathcal{S}_{\mathcal{G}}^f$ denote the generalized slice of $\mathcal{G}$.
\begin{itemize}
    \item A steady gradient generalized Ricci soliton $\mathcal{G}$ is called \emph{rigid} if there exists a neighborhood $\mathcal{U}$ in the space of generalized metrics $\mathcal{GM}$ such that $\mathcal{G}$ is the only element in $\mathcal{U}\cap \mathcal{S}_{\mathcal{G}}^f $.
    \item A 2-tensor $\gamma\in T_{\mathcal{G}}\mathcal{S}_{\mathcal{G}}^f$ is called an \emph{essential infinitesimal generalized solitonic deformation} of $\mathcal{G}$ if $\mathcal{R}'_{\mathcal{G}}(\gamma)=0$ where $T_{\mathcal{G}}\mathcal{S}_{\mathcal{G}}^f$ denotes the tangent space of the generalized slice $\mathcal{S}_{\mathcal{G}}^f$ at $\mathcal{G}$. 
    \item An essential infinitesimal generalized solitonic deformation $\gamma$ is \emph{integrable} if there exists a curve of steady gradient generalized Ricci solitons $\mathcal{G}(t)$ with $\mathcal{G}(0)=\mathcal{G}$ and $\frac{d}{dt}|_{t=0}\mathcal{G}(t)=\gamma$. 
\end{itemize}

\end{defn}

 In the following, we denote the set of essential infinitesimal generalized solitonic deformation to be $IGSD$. We furthermore obtain results on the rigidity of the steady gradient generalized Ricci soliton. The first step to approach the rigidity question is to discuss the existence of essential infinitesimal generalized solitonic deformations. We prove that 
\begin{theorem}\label{T2}
     Given any simply connected, compact, Bismut-flat manifold $(M,\mathcal{G})$. There exist essential infinitesimal generalized solitonic deformations of $\mathcal{G}$, and its dimension is no less than $n^2$. 
\end{theorem}

Our second step is to discuss the integrability of $IGSD$. We prove the following theorem which shows that it is equivalent to computing the $k$-th derivative of $\mathcal{R}_{\mathcal{G}}$ to check the integrability.

\begin{theorem}\label{T3}
    Let $\mathcal{G}$ be a steady gradient generalized Ricci soliton. There exists a neighborhood $\mathcal{U}$ of $\mathcal{G}$ in the generalized slice $\mathcal{S}^f_{\mathcal{G}}$ and a 
    finite-dimensional real analytic submanifold $\mathcal{Z}\subset \mathcal{U}$ such that 
    \begin{itemize}
        \item $T_{\mathcal{G}}(\mathcal{Z})=IGSD$,
        \item $\mathcal{Z}$ contains the premoduli space $\mathcal{P}_{\mathcal{G}}$ as a real analytic subset.
    \end{itemize}
\end{theorem}

In the last part of this work, we focus on the Bismut-flat case. The existence of the non-trivial $IGSD$ is proved in \Cref{T2}; however, it is hard to check its integrability. In a 3-dimensional case, we can explicitly construct the essential infinitesimal generalized solitonic deformations and prove that  
\begin{theorem}\label{T4}
    Suppose $(M,g,H)$ is a 3-dimensional Bismut-flat, Einstein manifold with positive Einstein constant $\mu$. Then, any essential infinitesimal generalized solitonic deformation is of the form
\begin{align*}
    \gamma=2\mu ug+\nabla^2 u-\frac{1}{2}d^*(uH)
\end{align*}
where $u$ is an eigenfunction with eigenvalue $4\mu$. Moreover, $\gamma$ is not integrable up to the second order if 
 \begin{align*}
\int_M \mu u^2w dV_g\neq 0
\end{align*}
for some eigenfunctions $w$ with eigenvalue $4\mu$.
\end{theorem}

In summary, in the 3-dimensional, Bismut--flat case, the dimension of the space of essential infinitesimal generalized solitonic deformation is 9. Some essential infinitesimal generalized solitonic deformations are integrable up to the second order while some are not. In \cite{classification} Corollary 1.4, Streets showed that there exists a non-trivial steady gradient generalized Ricci soliton in any dimension $n\geq3$. In particular, his proof showed that the Bismut-flat metric on $S^3$ is not rigid. Therefore, there exist an integrable essential infinitesimal generalized solitonic deformations.

The layout of this paper is as follows: in Section 2, we will mention some preliminaries regarding Courant algebroids, the generalized slice theorem, and the generalized Ricci solitons. In Section 3, we analyze the second variation formula of $\lambda$. In Section 4, we will discuss the essential infinitesimal generalized solitonic deformation and integrability properties. In Section 5, we will focus on the Bismut--flat case and provide some examples of essential infinitesimal generalized solitonic deformations.

\vspace{5pt}
\textbf{Acknowledgements:} This work is written when the author is a fourth-year math Ph.D. student at the University of California-Irvine. I am grateful to my advisor Jeffrey D.~Streets for his helpful advice. His suggestions play an important role in this work. 
\vspace{5pt}

\textbf{Declaration}: Availability of data and materials: Data sharing not applicable to this article as no data sets were generated
or analyzed during the current study.

\section{Preliminary}

\subsection{Notations}
In this work, we will use the following notation. Suppose  $h\in \Gamma(S^2M)$ and $K\in \Omega^2$,
\begin{align*}
   \mathring{R} (h)_{jk}= R_{ijkl}h_{il}, \quad \mathring{R}^+ (h)_{jk}= R^+_{ijkl}h_{il}, \quad \mathring{R} (K)_{jk}= R_{ijkl}K_{il}, \quad \mathring{R}^+ (K)_{jk}= R^+_{ijkl}K_{il},
\end{align*}
where $R^+$ denotes the Bismut curvature which is defined in \Cref{P4}. Besides, we will consider the $f$-twisted $L^2$ inner product 
\begin{align}
    \Big( \gamma_1,\gamma_2\Big)_{f}\coloneqq\int_M \langle \gamma_1,\gamma_2 \rangle_g  e^{-f} dV_g. \label{6}
\end{align}
where $\langle\kern0.3em,\kern0.3em\rangle_g$ denotes the standard inner product induced by a Riemannian metric $g$ and $\gamma_1,\gamma_2$ are tensors of same order. In particular, we mainly focus on the case when $\gamma$ is a 2-tensor. We then require the notation
\begin{align}
    \nonumber&\divg_f\omega=\nabla_i\omega_i-\nabla_i f\omega_i,\quad (\divg_fh)_i=\nabla_jh_{ji}-\nabla_jf h_{ji},
    \quad (\divg_f^*\omega)_{ij}=-\frac{1}{2}(\nabla_i\omega_j+\nabla_j\omega_i)=(\divg^*\omega)_{ij}
    \\&\triangle_f=\triangle-\nabla f\cdot\nabla, \quad d_f^*=d^*+i_{\nabla f},\quad  \text{where $\omega\in\Omega^1$ and $h\in \Gamma(S^2M)$.}
\end{align}
Later, we will use the different connection $\nabla^{\pm}$ which is defined in (\ref{13}). Thus, we will denote $\divg^{\pm}$ as the divergence operator computed by using connction $\nabla^{\pm}$ and $\divg^{\pm,*}$ as their formal adjoint.
\subsection{Generalized Geometry}
In this section, we review some basic definitions and properties of generalized geometry. More details can be found in \cite{GRF}.

\begin{defn}\label{D2}
A \emph{Courant algebroid} is a vector bundle $E\longrightarrow M$ with a nondegenerate bilinear form $\langle \cdot,\cdot\rangle$, a bracket $[\cdot,\cdot]$ on $\Gamma(E)$ and a bundle map $\pi:E\longrightarrow TM$ satisfies that for all $a,b,c\in\Gamma(E)$, $f\in C^\infty(M)$,
\begin{itemize}
    \item $[a,[b,c]]=[[a,b],c]+[b,[a,c]]$.
    \item $ \pi[a,b]=[\pi(a),\pi(b)].$
    \item $[a,fb]=f[a,b]+\pi(a)fb.$
    \item $\pi(a)\langle b,c\rangle=\langle [a,b],c \rangle+\langle a,[b,c] \rangle.$
    \item $[a,b]+[b,a]=\mathcal{D}\langle a,b \rangle $ where $\mathcal{D}: C^\infty(M)\to \Gamma(E)$ is given by $\mathcal{D}(\phi)\coloneqq\pi^*(d\phi)$.
\end{itemize}
We say a Courant algebroid $E$ is \emph{exact} if we have the following exact sequence of vector bundles 
\[  \begin{tikzcd}
  0 \arrow[r] & T^*M  \arrow[r, "\pi^*"] & E \arrow[r, "\pi"] & TM  \arrow[r] & 0 .
\end{tikzcd}
\]
\end{defn}

\begin{defn}\label{D3}
Let $E$ be an exact Courant algebroid. The \emph{automorphism group} $\Aut(E)$ of $E$ is a pair $(f,F)$ where $f\in\Diff(M)$ and $F\colon E\to E$ is a bundle map such that for all $u,v\in \Gamma(E)$
\begin{itemize}
    \item $ \langle Fu,Fv \rangle=f_{*}\langle u,v\rangle.$
    \item $[Fu,Fv]=F[u,v].$
    \item $\pi_{TM}\circ F=f_{*}\circ \pi_{TM}.$
\end{itemize}

\end{defn}

\begin{defn}\label{D4}
Given a smooth manifold $M$ and an exact Courant algebroid $E$ over $M$, a \emph{generalized metric} on $E$ is a bundle endomorphism $\mathcal{G}\in \Gamma(\End(E))$ satisfying 
\begin{itemize}
    \item $\langle \mathcal{G}a,\mathcal{G}b \rangle=\langle a,b\rangle.$
    \item $ \langle \mathcal{G}a,b \rangle=\langle a,\mathcal{G}b\rangle.$
    \item $ \langle \mathcal{G}a,b \rangle \text{  is symmetric and positive definite for any $a,b\in E$.}$
\end{itemize}

\end{defn}

\begin{example}
The most common and important example of Courant algebroids is $TM\oplus T^*M$. In this case, we define a nondegenerate bilinear form $\langle\cdot,\cdot\rangle$ and a bracket $[\cdot,\cdot]$ on $TM\oplus T^*M$ by
\begin{align*}
    \langle X+\xi,Y+\eta \rangle&\coloneqq\frac{1}{2}(\xi(Y)+\eta(X)),
    \\ [X+\xi,Y+\eta]_H&\coloneqq[X,Y]+L_X\eta-i_Y d\xi+i_Yi_XH
\end{align*}
where $X,Y\in TM$,  $\xi,\eta\in T^*M$ and $H$ is a 3-form. Define $\pi$ to be the standard projection, one can check that $(TM\oplus T^*M)_H\coloneqq(TM\oplus T^*M,\langle\cdot,\cdot\rangle,[\cdot,\cdot]_H,\pi)$ satisfies the Courant algebroid conditions. Moreover, its automorphism groups are given as follows.
\begin{align*}
   \GDiff_H=\{(f,\overline{f}\circ e^B): f\in\Diff(M), B\in \Omega^2 \text{  such that  } f^*H=H-dB\},
\end{align*}
where
\begin{align*}
    \overline{f}&=\begin{pmatrix} f_{\star} & 0 \\ 0 & (f^*)^{-1} \end{pmatrix}: X+\alpha\longmapsto f_{*}X+(f^*)^{-1}(\alpha),
    \\ e^B&=\begin{pmatrix} Id & 0 \\ B & Id \end{pmatrix}: X+\alpha\longmapsto X+\alpha+i_XB  \qquad \text{  for any $X\in TM$ and $\alpha\in T^*M$}.
\end{align*}
The product of automorphisms is given by
\begin{align*}
    (f,F)\circ (f',F')=\overline{f\circ f'}\circ e^{B'+f'^*B} \quad \text{  where } F=\overline{f}\circ e^B,\quad F'=\overline{f'}\circ e^{B'}.
\end{align*}
In the following, we will denote $\GDiff_H$ to be the automorphism group of $(TM\oplus T^*M)_H$, $\mathcal{GM}$ to be the space of all generalized metrics, and $\mathcal{M}$ to be the space of all Riemannian metrics. 
\end{example}

Recall that in \cite{GRF} Proposition 2.10, we see that for any exact Courant
Courant algebroid $E$ with a isotropic splitting $\sigma$, $E\cong_\sigma (TM\oplus T^*M)_H$ where 
\begin{align*}
    H(X,Y,Z)=2\langle [\sigma X,\sigma Y], \sigma Z \rangle \quad X,Y,Z \in TM. 
\end{align*}
Therefore, we see that
\begin{align*}
    \Aut(E)\cong_{\sigma}\GDiff_H=\{(\varphi,B)\in\Diff(M)\ltimes\Omega^2: \varphi^*H=H-dB\}.
\end{align*}
Moreover, we have the following proposition.

\begin{proposition}[\cite{GRF} Proposition 2.38 and 2.40] \label{P2}
Let $E$ be an exact Courant algebroid. The space of all generalized metrics $\mathcal{GM}$ on $E$ is isomorphic to $\mathcal{M}\times \Omega^2$.
\end{proposition}
\begin{remark}\label{R1}
Fix a background 3-form $H_0$ such that $E\cong (TM\oplus T^*M)_{H_0}$, the proof of \Cref{P2} implies that the 3-form $H$ of any generalized metric $\mathcal{G}=\mathcal{G}(g,b)$ is induced by an isotropic splitting $\sigma(X)=X+i_X b$ and then we have $H=H_0+db$. (See \cite{K} Remark 2.7 for more details.)   
\end{remark}

\subsection{Generalized Slice Theorem}

Recall that on $\mathcal{M}$, we have a natural group action which is given by 
\begin{align*}
    \rho_{\mathcal{M}}:\nonumber\quad &\Diff(M) \times \mathcal{M} \longrightarrow \mathcal{M}
    \\&(\varphi,g)\longmapsto \varphi^*g. 
\end{align*}
The quotient of $\mathcal{M}$ in terms of the action $\rho_{\mathcal{M}}$ is called the moduli space of Riemannian metrics. In order to study the moduli space of Riemannian metrics, Ebin proposed his slice theorem \cite{MR0267604} which proved the existence of a slice for the diffeomorphism group of a compact manifold acting on $\mathcal{M}$. 

In the generalized geometry, we define the $\GDiff_H$ action on generalized metrics by
\begin{align}
    \rho_{\mathcal{GM}}:\nonumber\quad &\GDiff_H \times \mathcal{GM} \longrightarrow \mathcal{GM}
    \\&((\varphi,B),(g,b))\longmapsto (\varphi^*g,\varphi^*b-B). \label{GA}
\end{align}
Here, we note that $\mathcal{GM}\cong \mathcal{M}\times \Omega^2 $ so in the following, we will always denote a generalized metric $\mathcal{G}$ by $\mathcal{G}(g,b)$ for some $(g,b)\in \mathcal{M}\times \Omega^2$. Therefore,
\begin{align*}
    T_{\mathcal{G}}\mathcal{GM}= T_{\mathcal{G}}(\mathcal{M}\times \Omega^2)=\Gamma(S^2M)\times \Omega^2=\otimes^2 T^*M.
\end{align*}
Naturally, we have a $L^2$ inner product on $T_{\mathcal{G}}\mathcal{GM}$ defined by
\begin{align}
    \Big(\gamma_1,\gamma_2\Big)=\Big((h_1,k_1),(h_2,k_2)\Big)\coloneqq\int_M \Big(\langle h_1,h_2 \rangle_g+\langle k_1,k_2 \rangle_g\Big)  dV_g  \label{I1}
\end{align}
where $\gamma_1=h_1-k_1,\gamma_2=h_2-k_2\in \otimes^2T^*M$ and   $ 
 (h_1,k_1),(h_2,k_2)\in \Gamma(S^2M)\times \Omega^2$.

In \cite{Rubio_2019}, Rubio and Tipler proposed the generalized Ebin's slice theorem based on the inner product (\ref{I1}). In \cite{K}, we proved the generalized slice theorem based on the $f$-twisted inner product (\ref{6}). The precise statement is as follows.
\begin{theorem}[\cite{K} Theorem 2.14]\label{T5}
Let $\mathcal{G}$ be a generalized metric on an exact Courant algebroid $E$ and $f$ be $\Isom(\mathcal{G})$ invariant, then there exists an submanifold $S^f_\mathcal{G}$ of $\mathcal{GM}$ such that 
\begin{itemize}
    \item $\forall F\in \Isom_H(\mathcal{G}),\, F\cdot S^f_\mathcal{G}=S^f_\mathcal{G}$.
    \item $ \forall F\in \GDiff_H, \text{if $(F\cdot S^f_\mathcal{G})\cap S^f_\mathcal{G}=\emptyset $ then }F\in \Isom_H(\mathcal{G})$.
    \item There exists a local cross section $\chi$ of the map $F\longmapsto \rho_{\mathcal{GM}}(F,\mathcal{G})$ on a neighborhood $\mathcal{U}$ of $\mathcal{G}$ in the orbit space $\mathcal{O}_{\mathcal{G}}=\GDiff_H\cdot\mathcal{G}$ such that the map from $\mathcal{U}\times S^f_{\mathcal{G}}\longrightarrow \mathcal{GM}$ given by $(V_1,V_2)\longmapsto \rho_{\mathcal{GM}}(\chi(V_1),V_2)$ is a homeomorphism onto its image.
\end{itemize}
where $\Isom_H(\mathcal{G})$ is the isotropy group of $\mathcal{G}$ under the $\GDiff_H$-action and it is called the group of generalized isometries of $\mathcal{G}\in\mathcal{GM}$. Moreover, the tangent space of the generalized slice on a generalized metric $\mathcal{G}(g,b)$ is given by
\begin{align}
    T_{\mathcal{G}}S^f_\mathcal{G}=\{\gamma=h-K\in\otimes^2T^*M: \kern0.5em (\divg_fh)_l=\frac{1}{2}K_{ab}H_{lab},  \kern0.5em d_f^*K=0 \}. 
\end{align}
\end{theorem}

\subsection{Bismut connection and curvature}
In this subsection, we aim to discuss the generalized Ricci flow and generalized Ricci solitons. Most of the contents can be found in \cite{GRF} Section 4.
 
\begin{defn}\label{D5}
Let $(M,g, H)$ be a Riemannian manifold and $H\in\Omega^3$. The \emph{Bismut connections} $\nabla^{\pm}$ associated to $(g,H)$ are defined as 
\begin{align}
    \langle \nabla_X^{\pm}Y,Z \rangle=\langle \nabla_XY,Z \rangle\pm \frac{1}{2}H(X,Y,Z) \quad \text{ for all tangent vectors $X,Y,Z$.} \label{13}
\end{align}
Here, $\nabla$ is the Levi-Civita connection associated with $g$, i.e., $ \nabla^{\pm}$ are the unique compatible connections with torsion $\pm H$. Later, we will mainly use $\nabla^+$.
\end{defn}

Following the definitions, we are able to compute the curvature tensor of the Bismut connection.
\begin{proposition}[\cite{GRF} Proposition 3.18]\label{P3}
Let $(M^n,g,H)$ be a Riemannian manifold with $H\in \Omega^3$ and $dH=0$, then for any vector fields $X,Y,Z,W$ we have
\begin{align*}
    \nonumber Rm^+(X,Y,Z,W)=& \kern0.2em Rm(X,Y,Z,W)+\frac{1}{2}\nabla_XH(Y,Z,W)-\frac{1}{2}\nabla_YH(X,Z,W)
    \\&-\frac{1}{4}\langle H(X,W),H(Y,Z)\rangle+\frac{1}{4}\langle H(Y,W),H(X,Z)\rangle, 
    \\\Rc^+&=\Rc-\frac{1}{4}H^2-\frac{1}{2}d^*H, \quad R^+=R-\frac{1}{4}|H|^2,
\end{align*}
where $H^2(X,Y)=\langle i_XH,i_YH\rangle$. Here $Rm^+$, $\Rc^+$, $R^+$ denote the Riemannian curvature, Ricci curvature, and scalar curvature with respect to the Bismut connection $\nabla^+$. In particular, if $(M,g,H)$ is Bismut-flat, then 
\begin{align*}
    Rm(X,Y,Z,W)=\frac{1}{4}\langle H(X,W),H(Y,Z)\rangle-\frac{1}{4}\langle H(Y,W),H(X,Z)\rangle\quad \text{and } \quad \nabla H=0
\end{align*}
for any vector fields $X,Y,Z,W$.
\end{proposition}

In \cite{J2}, the author defined general tensors regarding the Bismut connection. Later, we will see that these quantities are related to the generalized Einstein--Hilbert functional.
\begin{defn}\label{D6}
Given a metric $g$, closed three-form $H$, and smooth function $f$, a triple $(g,H,f)$ determines a \emph{twisted Bakry-Emery curvature}
\begin{align}
    \Rc^{H,f}=\Rc-\frac{1}{4}H^2+\frac{1}{2}\nabla^2 f -\frac{1}{2}(d^*H+i_{\nabla f}H)  \label{BC}
\end{align}
and a \emph{generalized scalar curvature}
\begin{align}
    R^{H,f}=R-\frac{1}{12}|H|^2+2\triangle f-|\nabla f|^2.\label{RC}
\end{align}
\end{defn}

As we mentioned in the introduction, we will consider a special connection $\overline{\nabla}$ on 2-tensors. 
\begin{defn}\label{D7}
Let $(M,g,H)$ be a Riemannian manifold, $H\in\Omega^3$ and $\nabla^{\pm}$ be the Bismut connections given in (\ref{13}). The \emph{mixed Bismut connection} is a connection $\overline{\nabla}$ on 2-tensors defined as follows. For any $\gamma\in \otimes^2T^*M$ and tangent vectors $X,Y,Z$, we have
\begin{align}
    (\overline{\nabla}_X\gamma)(Y,Z)=\nabla_X\Big(\gamma(Y,Z)\Big)-\gamma(\nabla^-_XY,Z)-\gamma(Y,\nabla^+_XZ). \label{MC}
\end{align}

\end{defn}

\begin{defn}
Let $\overline{\nabla}$ be the mixed connection. 
\begin{itemize}
    \item The \emph{$f$-twisted divergence operator} $\overline{\divg}_f:  \otimes^2T^*M\longrightarrow T^*M\times T^*M$ on 2-tensors with respect to $\overline{\nabla}$ is given by
\begin{align}
       (\overline{\divg}_f\gamma)_l=(\overline{\nabla}^m\gamma_{ml}-\nabla_mf \gamma_{ml},\overline{\nabla}^m\gamma_{lm}-\nabla_mf \gamma_{lm}),\label{div}
\end{align}
\item The \emph{$f$-twisted divergence operator} $\overline{\divg}_f:  T^*M\times T^*M\longrightarrow C^\infty(M)$ on $T^*M\times T^*M$ with respect to $\overline{\nabla}$ is given by
\begin{align}
       \overline{\divg}_f(u,v)=\frac{1}{2}(\divg_f u+\divg_f v). \label{div0}
\end{align}
\end{itemize}
In the following, we denote $\overline{\divg}_f^*$ to be the formal adjoint of the $f$-twisted divergence operator $\overline{\divg}_f$ with respect to $f$-twisted $L^2$ inner product (\ref{6}). 

\end{defn}

\begin{remark}
    Due to (\ref{MC}), it is not hard to see that 
\begin{align*}
    (\overline{\divg}_f\gamma)_l=(\overline{\nabla}^m\gamma_{ml}-\nabla_mf \gamma_{ml},\overline{\nabla}^m\gamma_{lm}-\nabla_mf \gamma_{lm})=\Big((\nabla^+)^m\gamma_{ml}-\nabla_mf \gamma_{ml},(\nabla^-)^m\gamma_{lm}-\nabla_mf \gamma_{lm}\Big).
\end{align*}
\end{remark}

\begin{defn}
The \emph{Laplace operator of the mixed Bismut connection} $\overline{\triangle}_f$ is defined by 
\begin{align}
    \overline{\triangle}_f=-\overline{\nabla}^{*_f}\overline{\nabla} \label{triangle}
\end{align}
where $\overline{\nabla}^{*_f}$ is the formal adjoint of $\overline{\nabla}$ with respect to $f$-twisted $L^2$ inner product (\ref{6}).    
\end{defn}

The following two lemmas provide us a detail information about $\overline{\divg}_f^*$ and $\overline{\triangle}_f$.

\begin{lemma}\label{div*}
Given any $(u,v)\in T^*M\times T^*M$,
\begin{align*}
    \overline{\divg}_f^*(u,v)_{ij}=-(\nabla^+)^iu_j-(\nabla^-)^jv_i.
\end{align*}
\end{lemma}
\begin{proof}
From the definition, we compute
\begin{align*}
    \int_M (\overline{\divg}_f \gamma)_l(u_l,v_l)e^{-f}dV_g&=\int_M \Big((\overline{\nabla}^m\gamma_{ml}-\nabla_m f \gamma_{ml})u_l+(\overline{\nabla}^m\gamma_{lm}-\nabla_m f\gamma_{lm})v_l\Big) e^{-f}dV_g
    \\&=\int_M \Big(((\nabla^+)^m\gamma_{ml}-\nabla_m f \gamma_{ml})u_l+((\nabla^-)^m\gamma_{lm}-\nabla_m f\gamma_{lm})v_l\Big) e^{-f}dV_g
    \\&=\int_M-\Big((\nabla^+)^mu_l \gamma_{ml}+(\nabla^-)^mv_l\gamma_{lm} \Big) e^{-f}dV_g
    \\&=\int_M-\gamma_{ml}\Big((\nabla^+)^mu_l+(\nabla^-)^lv_m\Big)e^{-f}dV_g.
\end{align*}
\end{proof}

\begin{lemma}\label{L1}
Given any 2-tensor $\gamma\in \otimes^2T^*M$,
\begin{align}
\overline{\triangle}_f\gamma_{ij}=\triangle_f\gamma_{ij}-H_{mjk}\nabla_m\gamma_{ik}+H_{mik}\nabla_m\gamma_{kj}-\frac{1}{4}(H_{jl}^2\gamma_{il}+H_{il}^2\gamma_{lj})-\frac{1}{2}H_{mkj}H_{mli}\gamma_{lk}. \label{20}  
\end{align}  

\end{lemma}
\begin{proof}
For any 2-tensor $\gamma\in \otimes^2T^*M$, using normal coordinates we have
\begin{align*}    \overline{\nabla}_m\gamma_{ij}=\nabla_m\gamma_{ij}-\frac{1}{2}H_{mjk}\gamma_{ik}+\frac{1}{2}H_{mik}\gamma_{kj}. 
\end{align*}
Then,
\begin{align*}
    \int_M |&\overline{\nabla}\gamma|^2 e^{-f}dV_g
    \\&=\int_M (\nabla_m\gamma_{ij}-\frac{1}{2}H_{mjk}\gamma_{ik}+\frac{1}{2}H_{mik}\gamma_{kj})(\nabla_m\gamma_{ij}-\frac{1}{2}H_{mjl}\gamma_{il}+\frac{1}{2}H_{mil}\gamma_{lj})e^{-f}dV_g
    \\&= \int_M \Big(|\nabla\gamma|^2+\nabla_m\gamma_{ij}(-H_{mjk}\gamma_{ik}+H_{mik}\gamma_{kj})+\frac{1}{4}(-H_{mjk}\gamma_{ik}+H_{mik}\gamma_{kj})(-H_{mjl}\gamma_{il}+H_{mil}\gamma_{lj})\Big)e^{-f}dV_g
    \\&=\int_M \Big(-\triangle_f\gamma_{ij}+H_{mjk}\nabla_m\gamma_{ik}-H_{mik}\nabla_m\gamma_{kj}+\frac{1}{4}(H_{jl}^2\gamma_{il}+H_{il}^2\gamma_{lj})+\frac{1}{2}H_{mkj}H_{mli}\gamma_{lk} \Big)\gamma_{ij}e^{-f}dV_g. 
\end{align*} 
Therefore, the result follows.

\end{proof}

\begin{proposition}\label{RTS}
Let $\mathcal{G}$ be a generalized metric on an exact Courant algebroid $E$ and $f$ be $\Isom(\mathcal{G})$ invariant. The tangent space of the generalized slice on a generalized metric $\mathcal{G}$ is given by
\begin{align}
    T_{\mathcal{G}}S^f_\mathcal{G}=\{\gamma\in\otimes^2T^*M: \overline{\divg}_f\gamma=0 \}. \label{TS}
\end{align}  
In the following, we say that a 2-tensor $\gamma$ is non-trivial if $ \gamma\in T_{\mathcal{G}}S^f_\mathcal{G}$.
\end{proposition}
\begin{proof}
 By (\ref{13}), we observe that 
\begin{align*}
    (\divg^{\pm}_fh)_l&=(\nabla^{\pm})^m h_{ml}-\nabla_mf h_{ml}=(\divg_f h)_l\mp\frac{1}{2}H_{mlk}h_{mk}=(\divg_f h)_l,
    \\(d_f^{\pm,*} K)_l&=-(\nabla^{\pm})^m K_{ml}+\nabla_mf K_{ml}=(d_f^*K)_l\mp\frac{1}{2}H_{mkl}K_{mk}.
\end{align*}
Then,
\begin{align*}
     (\nabla^+)^m\gamma_{ml}-\nabla_m f\gamma_{ml}&= (\divg^{+}_fh)_l+(d_f^{+,*} K)_l=(\divg_f h)_l+(d_f^*K)_l-\frac{1}{2}H_{mkl}K_{mk},
   \\ (\nabla^-)^m\gamma_{lm}-\nabla_m f\gamma_{lm}&= (\divg^{-}_fh)_l-(d_f^{-,*} K)_l=(\divg_f h)_l-(d_f^*K)_l-\frac{1}{2}H_{mkl}K_{mk}.
\end{align*}
Therefore,  our result follows by \Cref{T5}.

\end{proof}

\subsection{Generalized Ricci Solitons}
In this subsection, we discuss the generalized Ricci flow and generalized Ricci solitons. Most of the contents can be found in \cite{GRF} Section 4.

\begin{defn}\label{D8}
Let $E$ be an exact Courant algebroid over a smooth manifold $M$ and $H_0$ is a background closed 3-form. A one-parameter family of generalized metrics $\mathcal{G}_t=\mathcal{G}_t(g_t,b_t)$
is called a \emph{generalized Ricci flow} if 
\begin{align*}
   &\frac{\partial}{\partial t}g=-2\Rc+\frac{1}{2}H^2,\nonumber
   \\&\frac{\partial}{\partial t}b=-d^*H \quad \text{ where $H=H_0+db$.} 
\end{align*}
Equivalently, the generalized Ricci flow can also be expressed as
\begin{align*}
   &\frac{\partial}{\partial t}(g-b)=-2\Rc^+.
\end{align*}
where $\Rc^+$ denotes the Ricci curvature with respect to the Bismut connection $\nabla^+$.
\end{defn}

Motivated by the Ricci flow, we define the stationary points to be the steady generalized Ricci solitons. (For more details about motivations, readers can consult with \cite{GRF} and \cite{K}.)
\begin{defn}\label{D9}
Given a Riemannian metric $g$, a closed 3-form $H$, and a smooth function $f$ on a smooth manifold $M$. We say $(M,g,H,f)$ is a \emph{steady gradient generalized Ricci soliton} if 
\begin{align}
   0=\Rc-\frac{1}{4}H^2+\nabla^2 f, \quad 0=d_g^*H+i_{\nabla f}H\quad (\text{Equivalently, $\Rc^{H,f}=0$.}) \label{GRS}
\end{align}
and $(M,g,H)$ is called a \emph{generalized Einstein manifold} if
\begin{align*}
  0=\Rc-\frac{1}{4}H^2, \quad 0=d_g^*H.
\end{align*}
\end{defn}

\begin{example}
The most basic example of the steady generalized Ricci soliton is the work on $S^3$. Given a standard unit sphere metric $g_{S^3}$. By taking $H_{S^3}=2dV_{g_{S^3}}$, we get
\begin{align*}
    \Rc_{g_{S^3}}=\frac{1}{4}H^2_{S^3}, \quad d^*H_{S^3} =0,
\end{align*}
which implies that $(g_{S^3},H_{S^3} )$ is a generalized Einstein metric. Moreover, any 3-dimensional compact generalized Einstein manifold is a quotient of $S^3$. (See \cite{K}, Corollary 2.22 for more detail.)  
\end{example}

\begin{example}\label{E3}
 Suppose $G$ is a compact Lie group. By \cite{JM}, we know that $G$ possesses a bi-invariant metric $\langle\cdot ,\cdot\rangle$, and its corresponding connection, Riemann curvatures, sectional curvatures are given by
\begin{align*}
    \nabla_XY&=\frac{1}{2}[X,Y],\quad   R(X,Y)Z=-\frac{1}{4}[[X,Y],Z],
    \\ K(X,Y)&=\frac{1}{4}\langle [X,Y],[X,Y] \rangle \quad \text{where $X,Y,Z$ are left-invariant vector fields.}
\end{align*}
Define a 3-form $H$ by $g^{-1}H(X,Y)=[X,Y]$. By direct computation, we know that connections $\nabla^{\pm}$ are flat. Thus, any compact Lie group admits a Bismut-flat structure. 

On the other hand, a famous result of Cartan Schouter (See \cite{GRF} Theorem 3.54) shows that if $(M,g,H)$ is complete, simply connected and Bismut-flat then $(M, g)$ is isometric to a product of simple Lie groups with bi-invariant
metrics $g$, and $g^{-1}H(X, Y) = \pm[X, Y ]$ on left-invariant vector fields.

\end{example}

\subsection{Generalized Einstein--Hilbert functional}
In this subsection, we will see that generalized Ricci solitons are related to the generalized Einstein--Hilbert functional which are defined below.
\begin{defn}\label{D10}
Given a smooth manifold $M$ and a background closed 3-form $H_0$, the \emph{generalized Einstein--Hilbert functional} $\mathcal{F}: \Gamma(S^2M)\times \Omega^2 \times C^\infty (M)\to \mathbb{R}$ is given by 
\begin{align*}
    \mathcal{F}(g,b,f)=\int_M R^{H,f} e^{-f}dV_g=\int_M (R-\frac{1}{12}|H_0+db|^2+|\nabla f|^2)e^{-f}dV_g 
\end{align*}
where $R^{H,f}$ is the generalized scalar curvature given in (\ref{RC}) and $H=H_0+db$. Also, we define 
\begin{align*}
  \lambda(g,b)=\inf\Big\{\mathcal{F}(g,b,f)\big|\, f\in C^\infty(M),\,\int_M e^{-f}dV_g=1\Big \}.
\end{align*}

\end{defn}

Following the same argument in the Ricci flow case (see \cite{MR2302600} for more details), we can deduce that for any $(g,b)$, the minimizer $f$ is always achieved. Moreover, $\lambda$ satisfies that
\begin{align*}
    \lambda(g,b)=R-\frac{1}{12}|H_0+db|^2+2\triangle f-|\nabla f|^2 
\end{align*}
and it is the lowest eigenvalue of the Schrödinger operator $-4\triangle+R-\frac{1}{12}|H_0+db|^2.$ 

Let $\mathcal{G}_t(g_t, b_t)$ be a smooth family of generalized metrics on a smooth compact manifold $M$. Assume 
\begin{align*}
    \frac{\partial}{\partial t}\Big|_{t=0}(g_t-b_t) =\gamma=h-K, \quad  (g_0,b_0)=(g,b),
\end{align*}
The first variation formula of the generalized Einstein--Hilbert functional is given by
\begin{align}
    \nonumber\frac{d}{dt}\Big|_{t=0}\lambda(g_t,b_t)&=\int_M \Big[\langle -\Rc+\frac{1}{4}H^2-\nabla^2f,h \rangle-\frac{1}{2}\langle d^*H+i_{\nabla f}H,K\rangle \Big]e^{-f}dV_g
    \\&=\int_M -\langle \gamma, \Rc^{H,f} \rangle e^{-f}dV_g. \label{FV}
\end{align}
where $\Rc^{H,f}$ is the twisted Bakry-Emery curvature given in (\ref{BC}) and $\gamma=h-K$. Based on the first variation formula, we conclude that
\begin{corollary}\label{C2}
    The generalized metric $\mathcal{G}(g,b)$ is a critical point of $\lambda$ if and only if $(g,b,f)$ is a steady gradient generalized Ricci soliton with $f$ realizing the infinmum in the definition of $\lambda$.
\end{corollary}

\begin{remark}\label{R2}
Due to \Cref{C2}, we say that a generalized metric $\mathcal{G}(g,b)$ is a steady gradient generalized Ricci soliton if $(g,b,f)$ satisfies the equation (\ref{GRS}) where $f$ is the minimizer of $\lambda$.
\end{remark}

\section{Linear Stability}
\subsection{Analytic Properties of Generalized Einstein--Hilbert Functional}
Given a compact steady gradient generalized Ricci soliton $\mathcal{G})$. The main goal of this section is to compute the second variation formula. First, we recall that in \cite{K} we have an analyticity property.

\begin{proposition}[\cite{K} Proposition 5.5]\label{P4}
Let $\mathcal{G}(g_0,b_0)$ be a compact steady gradient generalized Ricci soliton. There exists a $C^{2,\alpha}$-neighborhood $\mathcal{U}$ of $(g_0,b_0)$ such that the minimizers $f_{(g,b)}$ depends analytically on $(g,b)$ and $\lambda(g,b)$ is an analytic function in $\mathcal{U}$.
\end{proposition}

The proof of \Cref{P4} is based on the implicit function theorem. Since compact steady gradient generalized Ricci solitons are critical points of $\lambda$-functional, we have the following results.
\begin{lemma}[\cite{K}, Lemma 3.6]\label{L2}
Given a compact steady gradient generalized Ricci soliton $\mathcal{G}(g,b)$. Suppose $\mathcal{G}_t(g_t,b_t)$ is a one-parameter family of generalized metrics and $f_t$ is the minimizer of $\lambda(g_t,b_t)$ such that 
\begin{align*}
    &\frac{\partial}{\partial t}\Big|_{t=0}(g_t-b_t) =\gamma=h-K, \quad  (g_0,b_0)=(g,b),
    \\& \frac{\partial}{\partial t}\Big|_{t=0}f_t =\phi, \quad  f_0=f.
\end{align*}
 We have
\begin{align*}
    \triangle_f(\tr_g h-2\phi)=\divg_f \divg_fh-\frac{1}{6}\langle dK,H \rangle. 
\end{align*}
\end{lemma}  

\begin{remark}
Recall in the proof of \Cref{RTS},  we have
\begin{align*}
    (\overline{\divg}_f\gamma)_l=\Big((\divg_f h)_l+(d_f^*K)_l-\frac{1}{2}H_{mkl}K_{mk},(\divg_f h)_l-(d_f^*K)_l-\frac{1}{2}H_{mkl}K_{mk}\Big).
\end{align*}
Thus,
\begin{align*}
    \overline{\divg}_f\overline{\divg}_f\gamma=\divg_f \divg_fh-\frac{1}{6}\langle dK,H \rangle. 
\end{align*}
Then, the variation $\phi$ can be viewed as a function of $\gamma$ which satisfies
\begin{align}
    \triangle_f(\tr_g \gamma-2\phi)=\overline{\divg}_f\overline{\divg}_f\gamma\quad \text{and} \quad \int_M (\tr_g \gamma-2\phi)e^{-f}dV_g=0. \label{Df}
\end{align}
In particular, if $\gamma \in  T_{\mathcal{G}}S^f_\mathcal{G}$, by (\ref{TS}) we have
\begin{align*}
    \tr_g\gamma=2\phi.
\end{align*}
\end{remark}

\subsection{Variation of Generalized Metrics}

In this subsection, we compute some variation formulas. 

\begin{lemma}\label{L3}
Given a compact steady gradient generalized Ricci soliton $\mathcal{G}(g,b)$. Suppose $\mathcal{G}(g_t,b_t)$ is a one-parameter family of generalized metrics and $f_t$ is the minimizer of $\lambda(g_t,b_t)$ such that 
\begin{align*}
    &\frac{\partial}{\partial t}\Big|_{t=0}(g_t-b_t) =\gamma=h-K, \quad  (g_0,b_0)=(g,b),
    \\& \frac{\partial}{\partial t}\Big|_{t=0}f_t =\phi, \quad  f_0=f.
\end{align*}
We have
\begin{align*}
    \nonumber&\kern-1em \frac{\partial}{\partial t}\Big|_{t=0} \Big(R_{ij}-\frac{1}{4} H^2_{ij}+\nabla_i\nabla_j f\Big)
    \\&= -\frac{1}{2}\triangle_f h_{ij}-(\mathring{R}^+ h)_{ij}-(\divg_f^* \divg_fh)_{ij}-\nabla_i\nabla_j(\frac{\tr_g h}{2}-\phi)+\frac{1}{8}(h_{jk}H^2_{ik}+h_{ik}H^2_{jk})+\frac{1}{4}h_{ac}H_{iab}H_{jcb}
    \\& \kern2em +\frac{1}{2}\nabla_lH_{ijk}h_{lk}-\frac{1}{4}\Big((dK)_{iab}H_{jab}+H_{iab}(dK)_{jab}\Big).
\end{align*}
In particular, if $\gamma=h-K \in  T_{\mathcal{G}}S^f_\mathcal{G}$, we have
\begin{align*}
    \nonumber&\kern-1em \frac{\partial}{\partial t}\Big|_{t=0} \Big(R_{ij}-\frac{1}{4} H^2_{ij}+\nabla_i\nabla_j f\Big)
    \\&= -\frac{1}{2}\triangle_f h_{ij}-(\mathring{R}^+ h)_{ij}+\frac{1}{8}(h_{jk}H^2_{ik}+h_{ik}H^2_{jk})+\frac{1}{4}h_{ac}H_{iab}H_{jcb}+\frac{1}{2}\nabla_l H_{ijk}h_{lk}
    \\& \kern2em-\frac{1}{2}\Big(\nabla_aK_{bi}H_{jab}+\nabla_aK_{bj}H_{iab}\Big)+\frac{1}{4}K_{ab}(\nabla_i H_{jab}+\nabla_j H_{iab}).
\end{align*}
\end{lemma}  
\begin{proof}
Recall that in \cite{K} Lemma 3.5 we computed the derivative in the general case. Here, we use \Cref{P3} to derive that
\begin{align*}
    \mathring{R}^+(h)_{ij}=\mathring{R}(h)_{ij}-\frac{1}{2}\nabla_kH_{ilj}h_{kl}-\frac{1}{4}H_{ilm}H_{jkm}h_{lk}.
\end{align*}
Replace $\mathring{R}$ by $\mathring{R}^+$, we get our first result. Now, we consider $\gamma=h-K \in T_{\mathcal{G}}S^f_\mathcal{G}$. Using (\ref{TS}) , we have
\begin{align*}
   (\divg_f^* \divg_fh)_{ij}&=-\frac{1}{2}\Big(\nabla_i( \divg_fh)_j+\nabla_j( \divg_fh)_i \Big)
   \\&=-\frac{1}{4}\Big(\nabla_iK_{ab}H_{jab}+\nabla_jK_{pq}H_{ipq}+K_{ab}(\nabla_iH_{jab}+\nabla_jH_{iab}) \Big).
\end{align*}
Therefore, we see that
\begin{align*}
    \nonumber&\kern-1em \frac{\partial}{\partial t}\Big|_{t=0} \Big(R_{ij}-\frac{1}{4} H^2_{ij}+\nabla_i\nabla_j f\Big)
    \\&= -\frac{1}{2}\triangle_f h_{ij}-(\mathring{R} h)_{ij}-(\divg_f^* \divg_fh)_{ij}-\nabla_i\nabla_j(\frac{\tr_g h}{2}-\phi)+\frac{1}{8}(h_{jk}H^2_{ik}+h_{ik}H^2_{jk})+\frac{1}{2}h_{ac}H_{iab}H_{jcb}
    \\& \kern2em-\frac{1}{4}\Big((dK)_{iab}H_{jab}+H_{iab}(dK)_{jab}\Big)
    \\&=-\frac{1}{2}\triangle_f h_{ij}-(\mathring{R}^+ h)_{ij}+\frac{1}{8}(h_{jk}H^2_{ik}+h_{ik}H^2_{jk})+\frac{1}{4}h_{ac}H_{iab}H_{jcb}+\frac{1}{2}\nabla_l H_{ijk}h_{lk}
    \\& \kern2em-\frac{1}{2}\Big(\nabla_aK_{bi}H_{jab}+\nabla_aK_{bj}H_{iab}\Big)+\frac{1}{4}K_{ab}(\nabla_i H_{jab}+\nabla_j H_{iab}).
\end{align*}
Here, we note that $\tr_gh=2\phi$.
\end{proof}

\begin{lemma}\label{L4}
Given a compact steady gradient generalized Ricci soliton $\mathcal{G}(g,b)$. Suppose $\mathcal{G}(g_t,b_t)$ is a one-parameter family of generalized metrics and $f_t$ is the minimizer of $\lambda(g_t,b_t)$ such that 
\begin{align*}
    &\frac{\partial}{\partial t}\Big|_{t=0}(g_t-b_t) =\gamma=h-K, \quad  (g_0,b_0)=(g,b),
    \\& \frac{\partial}{\partial t}\Big|_{t=0}f_t =\phi, \quad  f_0=f.
\end{align*}
We have
\begin{align*}
    \nonumber\kern-1em \frac{\partial}{\partial t}\Big|_{t=0} \Big(d^*H+i_{\nabla f}H\Big)_{ij}&= -\triangle_f K_{ij}-2\mathring{R}^+(K)_{ij}-\nabla_i(d_f^*K)_j+ \nabla_j(d_f^*K)_i+ h_{ml}\nabla_l H_{mij}
    \\&\kern2em+\frac{1}{4}H_{il}^2K_{lj}-\frac{1}{4}H_{jl}^2K_{li}+\frac{1}{2}(\nabla_i H_{klj}+\nabla_j H_{ikl})K_{kl}-\frac{1}{2}H_{ijm}H_{klm}K_{kl}+\frac{1}{2}H_{ilm}H_{jkm}K_{lk}
    \\&\kern2em+\Big((\divg_f h)_p-\frac{1}{2}\nabla_p (\tr_gh-2\phi)\Big)H_{pij}+\nabla_l h_{ip}H_{lpj}+\nabla_l h_{jp}H_{lip}.
\end{align*}
In particular, if $\gamma=h-K\in  T_{\mathcal{G}}S^f_\mathcal{G}$, we have
\begin{align*}
    \nonumber\kern-1em \frac{\partial}{\partial t}\Big|_{t=0} \Big(d^*H+i_{\nabla f}H\Big)_{ij}&= -\triangle_f K_{ij}-2\mathring{R}^+(K)_{ij}+\frac{1}{2}(\nabla_iH_{klj}+\nabla_jH_{kli})K_{kl}+h_{ml}\nabla_l H_{mij}
    \\&\kern2em+\frac{1}{2}K_{kl}H_{ikm}H_{jlm}+\frac{1}{4}H_{il}^2K_{lj}-\frac{1}{4}H_{jl}^2K_{li}+\nabla_l h_{ip}H_{lpj}+\nabla_l h_{jp}H_{lip}.
\end{align*}
\end{lemma}  
\begin{proof}
First, we note that the variation of $d^*H$ and $i_{\nabla f}H$ are given by
\begin{align*}
    \frac{\partial}{\partial t}\Big|_{t=0} (d^* H)_{ij}&=h_{ml}\nabla_l H_{mij}-\nabla_l(dK_{lij})+\Big((\divg h)_p-\frac{1}{2}\nabla_p \tr_gh\Big)H_{pij}+\nabla_l h_{ip}H_{lpj}+\nabla_l h_{jp}H_{lip}
\end{align*}
and
\begin{align*}
   \frac{\partial}{\partial t}\Big|_{t=0} (i_{\nabla f} H)_{ij}=\nabla_l\phi H_{lij}+\nabla_lf (dK)_{lij}-h_{lk}H_{kij}\nabla_lf.
\end{align*}
Therefore, 
\begin{align*}
    \nonumber&\kern-1em \frac{\partial}{\partial t}\Big|_{t=0} \Big(d^*H+i_{\nabla f}H\Big)_{ij}
    \\&= -\nabla_l(dK)_{lij}+ h_{ml}\nabla_l H_{mij}+\Big((\divg_f h)_p-\frac{1}{2}\nabla_p (\tr_gh-2\phi)\Big)H_{pij}+\nabla_l h_{ip}H_{lpj}+\nabla_l h_{jp}H_{lip}+\nabla_l f (dK)_{lij}.
\end{align*}
Then,
\begin{align*}
    -\nabla_l(dK)_{lij}&=-\nabla_l(\nabla_lK_{ij}+\nabla_iK_{jl}+\nabla_j K_{li})
    \\&=-\triangle K_{ij}-\nabla_l\nabla_iK_{jl}-\nabla_l\nabla_j K_{li}
    \\&=-\triangle K_{ij}-\nabla_i(d^*K)_j+ \nabla_j(d^*K)_i+R_{lijm}K_{ml}+R_{ljim}K_{lm}+(\Rc\circ K+K\circ \Rc)_{ij}
    \\&=-\triangle K_{ij}-\nabla_i(d^*K)_j+ \nabla_j(d^*K)_i-2\mathring{R}(K)_{ij}+(\Rc\circ K+K\circ \Rc)_{ij}
    \\&=-\triangle K_{ij}-\nabla_i(d_f^*K)_j+ \nabla_j(d_f^*K)_i-2\mathring{R}(K)_{ij}+\frac{1}{4}H_{il}^2K_{lj}-\frac{1}{4}H_{jl}^2K_{li}+\nabla_lf(\nabla_i K_{lj}-\nabla_j K_{li}),
\end{align*}
where we use the fact that $d_f^*K=d^*K+i_{\nabla f}K$. Similarly, we can replace the Riemann curvature with the Bismut Riemann curvature and deduce that
\begin{align*}
     \mathring{R}^+(K)_{ij}=\mathring{R}(K)_{ij}+\Big(\frac{1}{2}\nabla_i H_{klj}-\frac{1}{2}\nabla_k H_{ilj}-\frac{1}{4}H_{ijm}H_{klm}+\frac{1}{4}H_{ilm}H_{kjm}\Big)K_{kl}.
\end{align*}
Using the fact that $dH=0$, we have
\begin{align*}
    \nabla_kH_{ilj}K_{kl}=\frac{1}{2}(\nabla_iH_{klj}+\nabla_jH_{ilk})K_{kl}.
\end{align*}
Thus, we get our first result. If we consider $\gamma=h-K\in  T_{\mathcal{G}}S^f_\mathcal{G}$, our result follows by (\ref{TS}).

\end{proof}

\begin{proposition}\label{P5}
Given a compact steady gradient generalized Ricci soliton $\mathcal{G}(g,b)$. Suppose $\mathcal{G}(g_t,b_t)$ is a one-parameter family of generalized metrics and $f_t$ is the minimizer of $\lambda(g_t,b_t)$ such that 
\begin{align*}
    &\frac{\partial}{\partial t}\Big|_{t=0}(g_t-b_t) =\gamma=h-K, \quad  (g_0,b_0)=(g,b),
    \\& \frac{\partial}{\partial t}\Big|_{t=0}f_t =\phi, \quad  f_0=f.
\end{align*}
We have
\begin{align*}
     \frac{\partial }{\partial t}\Big|_{t=0}\Rc^{H,f}=-\frac{1}{2}\overline{\triangle}_f\gamma-\mathring{R}^+(\gamma)-\frac{1}{2}\overline{\divg}_f^*\overline{\divg}_f\gamma-\frac{1}{2}(\nabla^+)^2u.
\end{align*}
where $u=\tr_gh-2\phi$.
In particular, if $\gamma\in  T_{\mathcal{G}}S^f_\mathcal{G}$, we have
\begin{align*}
    \frac{\partial }{\partial t}\Big|_{t=0}\Rc^{H,f}=-\frac{1}{2}\overline{\triangle}_f\gamma-\mathring{R}^+(\gamma).
\end{align*}

\end{proposition}
\begin{proof}
Recall that $\Rc^{H,f}=\Rc-\frac{1}{4}H^2+\nabla^2f-\frac{1}{2}(d^*H+i_{\nabla f}H)$. \Cref{L3} and \Cref{L4} imply that   
\begin{align*}
    \frac{\partial }{\partial t}\Big|_{t=0}\Rc^{H,f}_{ij}&=-\frac{1}{2}\triangle_f \gamma_{ij}-\mathring{R}^+(\gamma)_{ij}+\frac{1}{4}\gamma_{ac}H_{iab}H_{jcb}+\frac{1}{8}(H_{il}^2\gamma_{lj}+\gamma_{il}H_{lj}^2)-\frac{1}{2}\nabla_a\gamma_{ib}H_{jab}+\frac{1}{2}\nabla_a\gamma_{bj}H_{iab}
    \\&\kern2em-(\divg_f^*\divg_f h)_{ij}+\frac{1}{2}d(d^*_f K)_{ij}-\frac{1}{4}\Big(\nabla_i(K_{ab}H_{jab})+\nabla_j(K_{ab}H_{iab})\Big)
    \\&\kern2em -\frac{1}{2}\Big((\divg_f h)_l-\frac{1}{2}K_{ab}H_{lab}\Big)H_{lij} -\nabla_i\nabla_j(\frac{\tr_gh}{2}-\phi)+\frac{1}{2}\nabla_l(\frac{\tr_gh}{2}-\phi)H_{lij}
    \\&=-\frac{1}{2}\overline{\triangle}_f\gamma_{ij}-\mathring{R}^+(\gamma)_{ij}-(\divg_f^*\alpha)_{ij}+\frac{1}{2}(d\zeta)_{ij}-\frac{1}{2}\alpha_l H_{lij}-\frac{1}{2}\nabla_i\nabla_ju+\frac{1}{4}\nabla_lu H_{lij}
    \\&=-\frac{1}{2}\overline{\triangle}_f\gamma_{ij}-\mathring{R}^+(\gamma)_{ij}+(\frac{\nabla_i^+\alpha_j+\nabla_j^-\alpha_i}{2})+(\frac{\nabla_i^+\zeta_j-\nabla_j^-\zeta_i}{2})-\frac{1}{2}\nabla^+_i\nabla_j^+u.
\end{align*}
where 
\begin{align*}
    \alpha_p\coloneqq (\divg_f h)_p-\frac{1}{2}K_{kl}H_{klp}, \quad \zeta\coloneqq d^*_fK \quad \text{and}\quad u\coloneqq\tr_gh-2\phi.
\end{align*}
Recall that in \Cref{RTS}, we deduced that
\begin{align*}
   (\divg_f\gamma)_l=(\alpha_l+\zeta_l,\alpha_l-\zeta_l).
\end{align*}
By \Cref{div*}, we conclude that 
\begin{align*}
     \frac{\partial }{\partial t}\Big|_{t=0}\Rc^{H,f}_{ij}=-\frac{1}{2}\overline{\triangle}_f\gamma_{ij}-\mathring{R}^+(\gamma)_{ij}-\frac{1}{2}\overline{\divg}_f^*\overline{\divg}_f\gamma_{ij}-\frac{1}{2}\nabla^+_i\nabla_j^+u.
\end{align*}
If $\gamma\in  T_{\mathcal{G}}S^f_\mathcal{G}$, then $\overline{\divg}_f\gamma$ vanish, giving the final claim.

\end{proof}

\subsection{Second variation formula}

Next, we analyze the second variation formula of the generalized Einstein--Hilbert functional.
\begin{theorem}\label{T6}
Given a compact steady gradient generalized Ricci soliton $\mathcal{G}(g,b)$ on a smooth manifold $M$. Suppose $\mathcal{G}(g_t,b_t)$ is a one-parameter family of generalized metrics such that 
\begin{align*}
    &\frac{\partial}{\partial t}\Big|_{t=0}(g_t-b_t) =\gamma=h-K, \quad  (g_0,b_0)=(g,b).
\end{align*}
The second variation of $\lambda$ on  $\mathcal{G}(g,b)$ is given by 
\begin{align*}
     \nonumber\frac{d^2}{dt^2}\Big|_{t=0}\lambda=\int_M \Big\langle \gamma, \frac{1}{2}\overline{\triangle}_f\gamma+\mathring{R}^+(\gamma)+\frac{1}{2}\overline{\divg}_f^*\overline{\divg}_f\gamma+\frac{1}{2}(\nabla^+)^2u \Big\rangle e^{-f}dV_g,
\end{align*} 
where $u$ is the unique solution of 
\begin{align*}
    \triangle_f u=\overline{\divg}_f\overline{\divg}_f\gamma,\quad \int_M ue^{-f}dV_g=0.
\end{align*}

\end{theorem}

\begin{proof}
The first variation formula (\ref{FV}) suggests that
\begin{align*}
    \frac{d}{dt}\lambda(g_t,b_t)=\int_M -\langle \gamma, \Rc^{H,f} \rangle e^{-f}dV_g.
\end{align*}
so
\begin{align*}
    \frac{d^2}{dt^2}\Big|_{t=0}\lambda(g_t,b_t)=\int_M -\langle \gamma, \frac{\partial }{\partial t}\Big|_{t=0}\Rc^{H,f} \rangle e^{-f}dV_g.
\end{align*}
By \Cref{P5} and \Cref{L2}, our result follows.

\end{proof}

\begin{corollary}\label{C3}
Every Bismut-flat, compact manifolds are linearly stable. The kernel of the second variation on a  Bismut-flat, compact manifold is non-trivial 2-tensor $\gamma$ which are parallel with respect to the connection $\overline{\nabla}$.
\end{corollary}
\begin{proof}
Due to the fact that $\lambda$ is diffeomorphism invariant and the generalized slice theorem, it suffices to show that 
\begin{align*}
     \frac{d^2}{dt^2}&\Big|_{t=0}\lambda(g_t,b_t)\leq 0 \quad \text{for all $\gamma=h-K\in T_{\mathcal{G}}S^f_{\mathcal{G}}$.}
\end{align*}
Therefore, the result follows by \Cref{T6}.
\end{proof}

\begin{remark}
Let $(M^n,\mathcal{G})$ be a compact generalized Einstein manifold. By \cite{K} Lemma 4.9, for any conformal variation $\gamma$, $\frac{\partial^2 }{\partial t^2}\lambda(\gamma)\leq 0$. Thus, it suffices to consider variation $\Tilde{\gamma}$ with 
\begin{align*}
    \overline{\divg}\Tilde{\gamma}=0,\quad \tr_g\Tilde{\gamma}=0.
\end{align*}
Recall that in Einstein manifold case, it suffices to consider variations which lie in $ TT_g=\{h\in\Gamma(S^2M), \divg h=0,\kern0.2em \tr_gh=0\}$ to discuss the linear stability. Therefore, our result matches our expectation.

\end{remark}

\subsection{The kernel variation of the second variation}

In this subsection, we assume that $(M,g,H)$ is a compact, Bismut-flat manifold. \Cref{C3} implies that $\gamma$ is a kernel of the second variation if $\gamma$ is non-trivial and parallel to the connection $\overline{\nabla}$. The lemma below suggests an idea to construct parallel variations.
\begin{lemma}\label{L5}
Suppose we have a $\nabla^-$-parallel 1-form $\alpha$ and a $\nabla^+$-parallel 1-form $\beta$. The 2-tensor $\gamma=\alpha\otimes\beta$ is non-trivial and parallel with respect to the connection $\overline{\nabla}$.
\end{lemma}
\begin{proof}
By assumption, we see that
\begin{align*}
    &\nabla_m\alpha_i=-\frac{1}{2}H_{mik}\alpha_k, \quad\nabla_m\beta_j=\frac{1}{2}H_{mjl}\beta_l.
\end{align*}
Then,
\begin{align*}
\nabla_m\gamma_{ij}=\nabla_m\alpha_i\otimes \beta_j+\alpha_i\otimes\nabla_m\beta_j=-\frac{1}{2}H_{mik}\gamma_{kj}+\frac{1}{2}H_{mjl}\gamma_{il}.
\end{align*}
which shows that $\overline{\nabla}\gamma=0$. Also, we have
\begin{align*}
    h_{ij}=\frac{1}{2}(\alpha_i\otimes\beta_j+\alpha_j\otimes\beta_i), \quad K_{ij}=\frac{1}{2}(-\alpha_i\otimes\beta_j+\alpha_j\otimes\beta_i).
\end{align*}
Therefore,
\begin{align*}
    (\divg h)_l&=\frac{1}{2}\nabla^k(\alpha_k\otimes\beta_l+\alpha_l\otimes\beta_k)=\frac{1}{4}H_{klm}(\alpha_k\otimes\beta_m- \alpha_m\otimes \beta_k )=\frac{1}{2}K_{mk}H_{mkl},
    \\(d^*K)_l&=\frac{1}{2}\nabla^m(\alpha_m\otimes\beta_l-\alpha_l\otimes\beta_m)=\frac{1}{4}H_{mlk}(\alpha_m\otimes \beta_k+\alpha_k\otimes \beta_l)=\frac{1}{2}H_{mlk}h_{mk}=0,
\end{align*}
i.e., $\gamma$ is non-trivial.
\end{proof}

\begin{corollary}\label{C4}
    Any compact Lie group is linearly stable and admits a non-trivial variation $\gamma$ such that its second variation vanishes. Moreover, the dimension of the kernel of the second variation is no less than $n^2$.
\end{corollary}
\begin{proof}
In \Cref{E3}, we know that any compact Lie group $G$ with a bi-invariant metric $g$ admits a Bismut-flat structure when we define a 3-form $H$ by $g^{-1}H(X,Y)=[X,Y]$ where $X,Y$ are left-invariant vector fields. Then, we consider the left-invariant coframe $\{\omega_1^L,\omega_2^L,...,\omega_n^L\}$ and the right-invariant coframe $\{\omega_1^R,\omega_2^R,...,\omega_n^R\}$. By definition, we see that each left-invariant one form is $\nabla^-$-parallel and right-invariant one form is $\nabla^+$-parallel. Therefore, corollary follows by \Cref{L5}. 
\end{proof}

\section{Infinitesimal Deformation}

In this section, we aim to discuss the infinitesimal deformations of gradient generalized Ricci solitons. Before we get started, let us review the results of Einstein manifolds. Most materials can be found in \cite{B} and a series of papers from Koiso \cite{Ko4,Ko1,Ko2, Ko5}. One can also consult with the work in the Ricci soliton case done by Kr{\"o}ncke Podest\`a, and Spiro \cite{K3,K4,K2,FA}.

\subsection{Infinitesimal Einstein Deformation}

Fix an Einstein metric $g$ on a manifold $M$. Let $\mathcal{M}_1$ denote the space of smooth metrics with unit volume. The moduli space of Einstein structures is the coset space $\mathcal{M}_1$ under the action of $\rho_{\mathcal{M}}$ endowed with the quotient topology. Naturally, we have a decomposition
\begin{align*}
    T_g\mathcal{M}_1=\im\divg^*\oplus (T_g\mathcal{M}_1\cap \ker\divg).
\end{align*}
Ebin's slice theorem suggests that there exists an analytic submanifold $\mathcal{S}_g\subset \mathcal{M}_1$ with $T_g \mathcal{S}_g=\ker\divg$. Then, we call the subset of Einstein metrics in $\mathcal{S}_g$ to be the premoduli space of Einstein structure around $g$.

Define the \emph{Einstein operator} $E$ by 
\begin{align*}
        E: \quad &\mathcal{M}_1\longrightarrow \Gamma(S^2M),
        \\& g\longmapsto \Rc_g-\frac{S(g)}{n}g,
\end{align*}
where $S(g)$ is the total scalar curvature functional which is defined as $S(g)=\int_M R_g dV_g$. In other words, Einstein metrics are the set $E^{-1}(0)$. Then, we say  $h\in \Gamma(S^2M)$ is an \emph{essential infinitesimal Einstein deformation} of an Einstein metric $g$ if 
\begin{align*}
    E'_g(h)=0 \quad \text{and} \quad h\in \ker \divg\cap\kern0.2em T_g\mathcal{M}_1.
\end{align*}
In the following, we denote the space of all infinitesimal Einstein deformation by $\epsilon(g)$ and the direct computation shows that
\begin{align*}
    \epsilon(g)=\{h: h\in\ker\divg, \kern0.5em \tr_gh=0, \kern0.5em \triangle h+2\mathring{R}(h)=0\}.
\end{align*}
An infinitesimal Einstein deformation $h$ is said integrable if there exists a $C^1$ curve of Einstein metrics $g(t)$ through $g=g_0$ such that $\frac{d}{dt}|_{t=0}g(t)=h$.

Given $h\in \epsilon(g)$, one of our questions is whether $h$ is integrable. To answer this question, \cite{B} Corollary 12.50  found that it is equivalent to check whether $h$ is formally integrable, i.e., if there exists $h_2,h_3,...$ such that $E(g(t))\equiv 0$ where
 \begin{align*}
       g(t)=g+th+\sum_{k=2}^\infty \frac{t^k}{k!}h_k.
\end{align*}
   
 Define the Bianchi operator $\beta_g$ by $\beta_g(h)=\divg h-\frac{1}{2}d\tr_gh$ where $h\in\Gamma(S^2M)$. It is not hard to see that $\beta_g(E(g))\equiv 0$ so the formal integrability is closely related to the Bianchi operator. In fact, it depends on the obstruction space $\ker\beta_g/ \im E_g'$. In \cite{B} Theorem 12.45, we have
\begin{align*}
    \ker\beta_g=\im E_g'\oplus \epsilon(g).
\end{align*}
Therefore, one sees that some infinitesimal deformations are not formally integrable. 

\subsection{Infinitesimal Generalized Solitonic Deformation}

\begin{defn}\label{D14}
    A steady gradient generalized Ricci soliton $\mathcal{G}_0(g_0,b_0)$ is called \emph{rigid} if there exists a neighborhood $\mathcal{U}\subset\mathcal{GM}$ such that for any steady gradient generalzied Ricci soliton $\mathcal{G}(g,b)\in\mathcal{U}$, there exists $(\varphi,B)\in\GDiff_{H_0}$ such that
    \begin{align*}
        (g,b)=\rho_{\mathcal{GM}}\Big((\varphi,B),(g_0,b_0)\Big)
    \end{align*}
where $\rho_{\mathcal{GM}}$ is given in (\ref{GA}).
\end{defn}
In other words, we say $\mathcal{G}_0$ is rigid if there exists some neighborhood $\mathcal{U}$ such that the generalized slice $\mathcal{S}_{\mathcal{G}_0}^{f_0}$ only contains one element $\mathcal{G}_0$. To study the local behavior, we define an operator
\begin{align*}
        \mathcal{R}: \quad &\mathcal{GM}\longrightarrow T^*M\otimes T^*M
        \\& \mathcal{G}(g,b)\longmapsto \Rc^{H,f}=\Rc-\frac{1}{4}H^2+\nabla^2f-\frac{1}{2}(d^*H+i_{\nabla}fH) 
    \end{align*}
where $f$ is the minimizer of $\lambda(g,b)$. Let's denote the space of steady gradient generalized Ricci solitons by $\mathcal{GRS}$ and we see that 
\begin{align*}
    \mathcal{GRS}=\mathcal{R}^{-1}(0).
\end{align*}

\begin{defn}\label{D15}
    Let $\mathcal{G}_0$ be a steady gradient generalized Ricci soliton. \emph{The premoduli space} of steady gradient generalized Ricci soliton at $\mathcal{G}_0$ is the set
    \begin{align*}
        \mathcal{P}_{\mathcal{G}_0}=\mathcal{GRS}\cap \mathcal{S}^{f_0}_{\mathcal{G}_0},
    \end{align*}
    where $ \mathcal{S}^{f_0}_{\mathcal{G}_0}$ is the generalized slice constructed in \Cref{T5}. 
\end{defn}

Locally, the map $\mathcal{R}$ is analytic and we are able to compute its derivative. By \Cref{P5}, we have the following results.

\begin{lemma}\label{L6}
 Let $\mathcal{A}$ denote the derivative of the operator $\mathcal{R}$. Then, $\mathcal{A}$ is given by   
\begin{align*}
 \mathcal{A}(\gamma)=-\frac{1}{2}\overline{\triangle}_f\gamma-\mathring{R}^+(\gamma)-\frac{1}{2}\overline{\divg}_f^*\overline{\divg}_f\gamma-(\nabla^+)^2u.
\end{align*}
where $u=\frac{\tr_gh}{2}-\phi$ and $\phi$ is the derivative of $f$ which is a function depending on $\gamma$. Moreover, $\mathcal{A}$ is a self-adjoint operator with respect to the inner product (\ref{6}).

\end{lemma}

Notice that $\mathcal{A}$ is not an elliptic operator. Motivated by \Cref{P5}, we then define an elliptic operator $\mathcal{B}$ by 
\begin{align}
    \mathcal{B}(\gamma)\coloneqq-\frac{1}{2}\overline{\triangle}_f \gamma-\mathring{R}^+(\gamma).
\end{align}

\begin{defn}\label{D16}
Let $\mathcal{G}$ be a steady gradient generalized Ricci soliton. 
\begin{itemize}
    \item  A 2-tensor $\gamma\in \otimes^2 T^*M$ is called an \emph{infinitesimal generalized solitonic deformation} of $\mathcal{G}$ if $\mathcal{A}(\gamma)=0$.
    \item A 2-tensor $\gamma\in \otimes^2 T^*M$ is called an \emph{essential infinitesimal generalized solitonic deformation} of $\mathcal{G}$ if $\mathcal{B}(\gamma)=0$ and $\gamma$ is non-trivial.
\end{itemize}
In the following, we denote the set of all essential infinitesimal generalized solitonic deformations by $IGSD$, i.e.
\begin{align*}
    IGSD=\{\gamma\in\otimes^2 T^*M: \mathcal{B}(\gamma)=0 \text{ and $\gamma$ is non-trivial.}\}.
\end{align*}

\end{defn}

Next, we define the Bianchi operator $\beta_{\mathcal{G}}$ by
\begin{align*}
    \mathcal{\beta}_{\mathcal{G}}: \quad & \otimes^2 T^*M\longrightarrow T^*M\times T^*M
        \\& \gamma\longmapsto \overline{\divg}_f\gamma=\Big( (\nabla^+)^m\gamma_{ml}-\nabla_m f\gamma_{ml},  (\nabla^-)^m\gamma_{lm}-\nabla_m f\gamma_{lm} \Big).
\end{align*}
Recall the proof of \Cref{RTS}, if $\gamma=h-K$, we could also write
\begin{align*}
    \beta_{\mathcal{G}}(\gamma)=\Big((\divg_f h)_l+(d_f^*K)_l-\frac{1}{2}H_{mkl}K_{mk},(\divg_f h)_l-(d_f^*K)_l-\frac{1}{2}H_{mkl}K_{mk}\Big).
\end{align*}
\begin{lemma}\label{L7}
    For any generalized metric $\mathcal{G}$, 
    \begin{align*}
        \beta_{\mathcal{G}}\circ\mathcal{R}(\mathcal{G})\equiv 0
    \end{align*}
\end{lemma}
\begin{proof}
Recall that $\mathcal{R}(\mathcal{G})=\Rc-\frac{1}{4}H^2+\nabla^2f-\frac{1}{2}(d^*H+i_{\nabla f}H)$, let
\begin{align*}
    h=\Rc-\frac{1}{4}H^2+\nabla^2f, \quad K=\frac{1}{2}d^*_fH.
\end{align*}
We compute
\begin{align*}
    \divg_f(\Rc-\frac{1}{4}H^2+\nabla^2f)_l&=\nabla_i(R_{il}-\frac{1}{4}H^2_{il}+\nabla_i\nabla_lf)-\nabla_i f (R_{il}-\frac{1}{4}H^2_{il}+\nabla_i\nabla_lf)
    \\&=\frac{1}{2}\nabla_l(R-|\nabla f|^2)+\triangle\nabla_lf-\nabla_i f R_{il}-\frac{1}{4}\nabla_i(H_{iab}H_{lab})+\frac{1}{4}\nabla_i f  H_{il}^2
    \\&=\frac{1}{2}\nabla_l(R+2\triangle f-|\nabla f|^2)+\frac{1}{4}(d^*_fH)_{ab}H_{lab}-\frac{1}{4}H_{iab}\nabla_i H_{lab}.
\end{align*}
Using the fact that $dH=0$, we have
\begin{align*}
    H_{iab}\nabla_i H_{lab}&=H_{iab}(\nabla_lH_{iab}-\nabla_aH_{ilb}+\nabla_bH_{ila})
    \\&=\frac{1}{2}\nabla_l|H|^2-2H_{iab}\nabla_aH_{ilb}
    \\&=\frac{1}{2}\nabla_l|H|^2-2H_{iab}\nabla_iH_{alb}.
\end{align*}
So
\begin{align*}
     \divg_f(\Rc\frac{1}{4}H^2+\nabla^2f)_l=\frac{1}{2}\nabla_l(R-\frac{1}{12}|H|^2+2\triangle f-|\nabla f|^2)+\frac{1}{4}(d^*_fH)_{ab}H_{lab}.
\end{align*}
Recall that $\lambda=R-\frac{1}{12}|H|^2+2\triangle f-|\nabla f|^2$ is a constant, then $(\divg_fh)_l-\frac{1}{2}K_{ij}H_{ijl}=0$. Our result follows by the fact
\begin{align*}
    d^*_f(d^*_fH)=0.
\end{align*}

\end{proof}

\begin{proposition}\label{P6}
    For any steady gradient generalized Ricci soliton $\mathcal{G}$, we have a decomposition
    \begin{align*}
        \ker \beta_{\mathcal{G}}=\im\mathcal{A}\oplus IGSD.
    \end{align*}
\end{proposition}
\begin{proof}
Since $\gamma$ is non-trivial if $\gamma\in IGSD$, it is obvious that $IGSD\subset \ker\beta_{\mathcal{G}}$. Also, \Cref{L7} implies that $\im\mathcal{A}\subset \ker\beta_{\mathcal{G}}$ since $\mathcal{R}(\mathcal{G})\equiv 0$. On the other hand, because $\mathcal{A}$ is self adjoint, we have
\begin{align*}
    \otimes^2T^*M=\ker\mathcal{A}\oplus \im\mathcal{A}.
\end{align*}
For any $\gamma\in\ker\beta_{\mathcal{G}}$, we write
\begin{align*}
    \gamma=\mathcal{A}(\gamma_1)+\gamma_2 \quad \text{where $\gamma_2\in\ker\mathcal{A}$.}
\end{align*}
Then,
\begin{align*}
    0=\beta_\mathcal{G}(\gamma)=\beta_\mathcal{G}(\gamma_2)
\end{align*}
which implies that $\gamma_2$ is non-trivial. Therefore, we finish the proof.
\end{proof}

\begin{lemma}\label{L8}
Define an elliptic operator
\begin{align*}
        \Phi: \quad &T^*M\oplus T^*M\longrightarrow T^*M\oplus T^*M
        \\& (u,v)\longmapsto (-\frac{1}{2}\triangle^+_f u,-\frac{1}{2}\triangle_f^-v).
\end{align*}
For any steady gradient generalized Ricci soliton $\mathcal{G}$ and $\gamma\in\otimes^2 T^*M$, we have
\begin{align*}
    \beta_\mathcal{G}(\mathcal{B}(\gamma))=\Phi(\overline{\divg}_f\gamma).
\end{align*}

\end{lemma}
\begin{proof}
By \Cref{L7}, we see that if $\mathcal{G}\in\mathcal{GRS}$, 
\begin{align*}
    \beta_{\mathcal{G}}\circ\mathcal{R}'_{\mathcal{G}}(\gamma)=0
\end{align*}
for any $\gamma\in\otimes^2T^*M$. From the definition, we see that
\begin{align*}
    (\mathcal{B}(\gamma)-\mathcal{A}(\gamma))_{ij}&=\frac{1}{2}(\overline{\divg}_f^*\overline{\divg}_f\gamma)_{ij}+\nabla^+_i\nabla^+_ju
    \\&=(\divg_f^*\alpha)_{ij}-\frac{1}{2}(d\zeta)_{ij}+\frac{1}{2}\alpha_l H_{lij}+\nabla_i\nabla_ju-\frac{1}{2}\nabla_lu H_{lij}
\end{align*}
where 
\begin{align*}
     \alpha_p= (\divg_f h)_p-\frac{1}{2}K_{kl}H_{klp}, \quad \zeta= d^*_fK, \quad \text{and} \quad u=\frac{\tr_gh}{2}-\phi.
\end{align*}
For convenience, let
\begin{align*}
    \Tilde{h}_{ij}=(\divg_f^*\alpha)_{ij}+\nabla_i\nabla_ju \quad \text{and}\quad \Tilde{K}_{ij}=\frac{1}{2}(d\zeta)_{ij}-\frac{1}{2}\alpha_lH_{lij}-\frac{1}{2}\nabla_lu H_{lij}.
\end{align*}
Therefore, 
\begin{align*}
     \beta_\mathcal{G}(\mathcal{B}(\gamma))=\beta_\mathcal{G}(\mathcal{B}(\gamma)-\mathcal{A}(\gamma))=\Big( (\divg_f \Tilde{h})_l+(d_f^*\Tilde{K})_l-\frac{1}{2}\Tilde{K}_{ab}H_{lab},\kern0.5em (\divg_f \Tilde{h})_l-(d_f^*\Tilde{K})_l-\frac{1}{2}\Tilde{K}_{ab}H_{lab}\Big).
\end{align*}
We compute
\begin{align*}
    \divg_f(\nabla^2u)_l&=\nabla^k(\nabla_k\nabla_l u)-\nabla_kf \nabla_k\nabla_lu
    \\&=\nabla_l\triangle u+R_{kl}\nabla_k u-\nabla_l(\nabla_k u\nabla_k f)+\nabla_ku \nabla_l\nabla_k f
    \\&=\nabla_l(\triangle_f u)+\frac{1}{4}H_{kl}^2\nabla_k u
    \\&=\frac{1}{2}\nabla_l(\divg_f \alpha)+\frac{1}{4}H_{kl}^2\nabla_k u.
\end{align*}
Also,
\begin{align*}
    (\divg_f\Tilde{h})_l&=\nabla^k(\divg^*_f\alpha)_{kl}-\nabla_k f(\divg^*_f\alpha)_{kl}+\frac{1}{2}\nabla_l(\divg_f \alpha)+\frac{1}{4}H_{kl}^2\nabla_k u
    \\&=-\frac{1}{2}\triangle \alpha_l-\frac{1}{8}H_{lm}^2\alpha_m-\nabla_k f (\divg^*_f\alpha)_{kl}-\frac{1}{2}\nabla_kf \nabla_l\alpha_k+\frac{1}{4}H_{lk}^2\nabla_ku
    \\&=-\frac{1}{2}\triangle_f \alpha_l-\frac{1}{8}H_{lm}^2\alpha_m+\frac{1}{4}H_{lk}^2\nabla_ku.
\end{align*}
Furthermore,
\begin{align*}
    \frac{1}{2}\Tilde{K}_{ij}H_{ijl}&=\frac{1}{4}(d\zeta)_{ij}H_{ijl}-\frac{1}{4}\alpha_p H^2_{pl}+\frac{1}{4}H_{pl}^2\nabla_pu,
    \\(d^*_f\Tilde{K})_l&=\frac{1}{2}(d^*_fd\zeta)_l+\frac{1}{2}\nabla_m\alpha_p H_{pml}.
\end{align*}
Here, we note that by \Cref{L2} 
\begin{align*}
    \divg_f\alpha=2\triangle_f u.
\end{align*}
In addition,
\begin{align*}
    (d_f^*d\zeta)_l&=-\nabla^m(d\zeta)_{ml}+\nabla_mf(d\zeta)_{ml}
    \\&=-\nabla^m(\nabla_m\zeta_l-\nabla_l\zeta_m)+\nabla^m f(\nabla_m\zeta_l-\nabla_l\zeta_m)
    \\&=-\triangle_f\zeta_l+\nabla_m\nabla_l\zeta_m-\nabla_mf\nabla_l\zeta_m
    \\&=-\triangle_f\zeta_l+\nabla_l\nabla_m\zeta_m+R_{lm}\zeta_m-\nabla_mf\nabla_l\zeta_m
    \\&=-\triangle_f\zeta_l+\frac{1}{4}H_{lm}^2\zeta_m,
\end{align*}
where we use the fact that $d_f^*\zeta=0$. Therefore, we compute that 
\begin{align*}
    (\divg_f\Tilde{h})_l-\frac{1}{2}\Tilde{K}_{ij}H_{ijl}\pm (d^*_f\Tilde{K})_l&=-\frac{1}{2}\triangle_f \alpha_l+\frac{1}{8}H_{lm}^2\alpha_m-\frac{1}{4}(d\zeta)_{ij}H_{ijl}\pm (\frac{1}{2}(d^*_fd\zeta)_l+\frac{1}{2}\nabla_m\alpha_p H_{pml})
    \\&=-\frac{1}{2}\triangle_f \alpha_l+\frac{1}{8}H_{lm}^2\alpha_m-\frac{1}{2}\nabla_i\zeta_j H_{ijl}\pm (-\frac{1}{2}\triangle_f\zeta_l+\frac{1}{8}H_{lm}^2\zeta_m+\frac{1}{2}\nabla_m\alpha_p H_{pml})
    \\&=-\frac{1}{2}\triangle_f(\alpha_l\pm\zeta_l)+\frac{1}{8}H_{lm}^2(\alpha_m\pm\zeta_m)-\frac{1}{2}\nabla_i(\zeta_j\pm \alpha_j)H_{ijl}.
\end{align*}
By definition, $u_l=\alpha_l+\zeta_l$ and $v_l=\alpha_l-\zeta_l$. We note that 
\begin{align*}
    \triangle_f^+ u_l=\triangle_fu_l+H_{mkl}\nabla_mu_k-\frac{1}{4}H_{jl}^2 u_j, \quad  \triangle_f^- v_l=\triangle_fv_l-H_{mkl}\nabla_mv_k+\frac{1}{4}H_{jl}^2 v_j
\end{align*}
so the lemma follows.
\end{proof}

\subsection{Integrability}

\textbf{Notations:} In the following, we write the formal power series expansion of generalized metrics $\mathcal{G}(t)=\mathcal{G}(g(t),b(t))$ starting at $\mathcal{G}(g,b)$ by
\begin{align*}
    \mathcal{G}(t)=\mathcal{G}+\sum_{l=1}^\infty \frac{\gamma_l}{l!}t^l.
\end{align*}
It means that $(g(t),b(t))$ is expanded by
\begin{align*}
    g(t)=g+\sum_{l=1}^\infty \frac{h_l}{l!}t^l \quad \text{and}\quad b(t)=b+\sum_{l=1}^\infty \frac{K_l}{l!}t^l
\end{align*}
where $h_l\in \Gamma(S^2M)$, $K_l\in \Omega^2$ and $\gamma_l=h_l-K_l$ for $l=1,2,3,...$. In other words, this notation means that if 
\begin{align*}
    h_l=\frac{d^l}{dt^l}\Big|_{t=0} g(t)  \quad \text{and}\quad  K_l=\frac{d^l}{dt^l}\Big|_{t=0} b(t), 
\end{align*}
we denote
\begin{align*}
    \gamma_l=\frac{d^l}{dt^l}\Big|_{t=0} \mathcal{G}(t). 
\end{align*}

\begin{defn}\label{D17}
Let $\mathcal{G}$ be a steady gradient generalized Ricci soliton. 
\begin{itemize}
    \item An essential infinitesimal generalized solitonic deformation $\gamma_1$ is \emph{formally integrable} if there exists 2-tensors $\gamma_2,\gamma_3,....$ such that the formal power series $\mathcal{G}(t)=\mathcal{G}+\sum_{k=1}^\infty \frac{\gamma_k}{k!}t^k$ satisfies 
    \begin{align*}
        \mathcal{R}(\mathcal{G}(t))\equiv0
    \end{align*}
    If we denote the formal power series of $\mathcal{R}(\mathcal{G}(t))$ by $\mathcal{R}(\mathcal{G}(t))=\mathcal{R}(\mathcal{G})+\sum_{k=1}^\infty \frac{\mathcal{R}_\mathcal{G}^{(k)}(\gamma_1,\gamma_2,...,\gamma_k)}{k!}t^k$. It is equivalent to say that $\mathcal{R}_\mathcal{G}^{(k)}(\gamma_1,\gamma_2,...,\gamma_k)=0$ for all $k=1,2,....$ In particular, if there exists $\gamma_2,\gamma_3,...,\gamma_l$ for some finite integer $l$ such that $\mathcal{R}_\mathcal{G}^{(m)}(\gamma_1,\gamma_2,...,\gamma_m)=0$ for $m=1,2,...,l$, then we say that $\gamma_1$ is \emph{formally integrable} up to order $l$.
\item An essential infinitesimal generalized solitonic deformation $\gamma$ is \emph{integrable} if there exists a curve of steady gradient generalized Ricci solitons $\mathcal{G}(t)$ with $\mathcal{G}(0)=\mathcal{G}$ and $\frac{d}{dt}|_{t=0}\mathcal{G}(t)=\gamma$.

\end{itemize}
\end{defn}

\begin{theorem}\label{T7}
    Let $\mathcal{G}(g,b)$ be a steady gradient generalized Ricci soliton. There exists a neighborhood $\mathcal{U}$ of $\mathcal{G}(g,b)$ in the slice $\mathcal{S}^f_{\mathcal{G}}$ and a 
    finite-dimensional real analytic submanifold $\mathcal{Z}\subset \mathcal{U}$ such that 
    \begin{itemize}
        \item $T_{\mathcal{G}}(\mathcal{Z})=IGSD$
        \item $\mathcal{Z}$ contains the premoduli space $\mathcal{P}_{\mathcal{G}}$ as a real analytic subset.
    \end{itemize}
\end{theorem}
\begin{proof}
Due to \cite{Ko4} Lemma 13.6, it suffices to  show that $\{\mathcal{R}'_{\mathcal{G}}(\gamma): \gamma\in T_{\mathcal{G}}\mathcal{S}_{\mathcal{G}}^f\}$ is closed. Note that 
\begin{align*}
    \{\mathcal{R}'_{\mathcal{G}}(\gamma): \gamma\in T_{\mathcal{G}}\mathcal{S}_{\mathcal{G}}^f\}= \{\mathcal{B}(\gamma): \gamma\in T_{\mathcal{G}}\mathcal{S}_{\mathcal{G}}^f\}.
\end{align*}
We aim to check that the set $\{\mathcal{B}(\gamma): \gamma\in T_{\mathcal{G}}\mathcal{S}_{\mathcal{G}}^f\}$ is closed. By \Cref{L8}, we see that
\begin{align*}
    \beta_\mathcal{G}(\mathcal{B}(\gamma))=\Phi(\overline{\divg}_f\gamma).
\end{align*}
If $\gamma\in T_{\mathcal{G}}\mathcal{S}_{\mathcal{G}}^f$, $\overline{\divg}_f\gamma=0$ so 
\begin{align*}
    \{\mathcal{B}(\gamma): \gamma\in T_{\mathcal{G}}\mathcal{S}_{\mathcal{G}}^f\}\subseteq \ker\beta\cap \im\mathcal{B}.
\end{align*}
Note that $\mathcal{B}$ is an elliptic operator then $\ker\beta\cap \im\mathcal{B}$ is a closed subset in $\ker\beta$. On the other hand, for any $\mathcal{B}(\gamma)\in \ker\beta$, we can decompose $\gamma=\gamma_1+\gamma_2$ where $\gamma_1\in T_{\mathcal{G}}\mathcal{S}_{\mathcal{G}}^f$ and $\gamma_2\in T_{\mathcal{G}}\mathcal{O}_{\mathcal{G}}$ by the generalized slice theorem. Then,
\begin{align*}
    0=\beta(\mathcal{B}(\gamma))=\Phi(\overline{\divg}_f\gamma_2).
\end{align*}
Since $\Phi$
is elliptic, it implies that $\{\mathcal{B}(\gamma): \gamma\in T_{\mathcal{G}}\mathcal{S}_{\mathcal{G}}^f\}$ has finite codimension in the closed subset $\ker\beta\cap \im\mathcal{B}$. Follow the argument in \cite{10.2307/j.ctt1bd6k0x} page 119, one can prove that $\{\mathcal{B}(\gamma): \gamma\in T_{\mathcal{G}}\mathcal{S}_{\mathcal{G}}^f\}$ is closed.

\end{proof}

\begin{corollary}\label{C5}
Let $\mathcal{G}$ be a steady gradient generalized Ricci soliton. An essential infinitesimal generalized solitonic deformation is integrable if and only if it is formally integrable.    
\end{corollary}

\begin{corollary}\label{C6}
Let $\mathcal{G}$ be a steady gradient generalized Ricci soliton. $\mathcal{G}$ is rigid if every essential infinitesimal solitonic deformation at $\mathcal{G}$ is integrable up to finite order.  
\end{corollary}

In general, it is hard to check whether $\gamma\in IGSD$ is integrable or not. The following result provides us a condition to check whether $\gamma$ is integrable up to the second order.
\begin{lemma}\label{L9}
    Suppose $\gamma$ is an infinitesimal generalized Ricci solitonic deformation on a steady gradient generalized Ricci soliton $\mathcal{G}$. Then, $\gamma$ is integrable up to the second order if and only if $\mathcal{R}''(\gamma,\gamma)$ is orthogonal to IGSD.
\end{lemma}

\begin{proof}
  Consider a curve of generalized metrics $\mathcal{G}(t)$ with $\mathcal{G}(0)=\mathcal{G}$ and $\frac{d}{dt}|_{t=0}\mathcal{G}(t)=\gamma$. Then,
  \begin{align*}
      &0=\frac{d}{dt}|_{t=0}\mathcal{R}(\mathcal{G}(t))=\mathcal{R}'(\mathcal{G}'(0)),
      \\&0=\frac{d^2}{dt^2}|_{t=0}\mathcal{R}(\mathcal{G}(t))=\mathcal{R}''(\mathcal{G}'(0),\mathcal{G}'(0))+\mathcal{R}'(\mathcal{G}''(0)).
  \end{align*}
Therefore, $\gamma$ is integrable up to second order if and only if $\mathcal{R}''(\gamma,\gamma)\in \im \mathcal{A}$. Since $\beta_{\mathcal{G}}(\mathcal{R}(\mathcal{G}))=0$,
\begin{align*}
    0=\beta''_{\mathcal{G}}(\gamma,\gamma)(\mathcal{R}(\mathcal{G}))+2\beta'_{\mathcal{G}}(\mathcal{R}'(\mathcal{G}))+\beta_{\mathcal{G}}(\mathcal{R}''(\gamma,\gamma))=\beta_{\mathcal{G}}(\mathcal{R}''(\gamma,\gamma)).
\end{align*}
The result is followed by \Cref{P6}.

\end{proof}

\subsection{Infinitesimal generalized Einstein deformations}
Similar to the steady 
 gradient generalized Ricci soliton case, we define 
 we define an operator
\begin{align*}
        \mathcal{E}: \quad &\mathcal{GM}\longrightarrow T^*M\otimes T^*M
        \\& \mathcal{G}(g,b)\longmapsto \Rc-\frac{1}{4}H^2-\frac{1}{2}d^*H 
\end{align*}
We denote the space of generalized Einstein metrics with unit volume to be $\mathcal{GE}$  and then
\begin{align*}
\mathcal{GE}=\mathcal{E}^{-1}(0).
\end{align*}
Let $\mathcal{G}_0$ be a generalized Einstein metric. The premoduli space of generalized Einstein metrics at $\mathcal{G}_0$ is 
\begin{align*}
    \Tilde{\mathcal{P}}_{\mathcal{G}_0}=\mathcal{GE}\cap \mathcal{S}^{0}_{\mathcal{G}_0},
\end{align*}
where $\mathcal{S}^{0}_{\mathcal{G}_0}$ is the generalized slice (Note that the minimizer is 0 in this case.). Similarly, we can take the derivative of $\mathcal{E}$. By (\ref{Df}), we see that the deformation $\gamma$ satisfies
\begin{align}
    \triangle(\tr_gh)=\divg\divg h-\frac{1}{6}\langle dK,H \rangle, \quad \int_M \tr_gh dV_g=0. \label{c}
\end{align}
Therefore, we have the following definition.

\begin{defn}\label{D18}
Let $\mathcal{G}$ be a generalized Einstein metric ($f=0$). 
\begin{itemize}
    \item A 2-tensor $\gamma\in \otimes^2 T^*M$ is called an \emph{infinitesimal generalized Einstein deformation} of $\mathcal{G}$ if $\mathcal{A}(\gamma)=0$ and $\gamma=h-K$ satisfies (\ref{c}).
    \item A 2-tensor $\gamma\in \otimes^2 T^*M$ is called an \emph{essential infinitesimal generalized Einstein deformation} of $\mathcal{G}$ if $\mathcal{B}(\gamma)=0$ and $\gamma=h-K\in T_{\mathcal{G}}\mathcal{S}_{\mathcal{G}}^0$ with $\tr_gh=0$.
\end{itemize}
In the following, we denote the set of all essential infinitesimal generalized Einstein deformations by $IGED$, i.e.
\begin{align*}
    IGED=\{\gamma\in\otimes^2 T^*M: \mathcal{B}(\gamma)=0, \gamma=h-K\in T_{\mathcal{G}}\mathcal{S}_{\mathcal{G}}^0 \text{ and $\tr_gh=0$.}\}.
\end{align*}
\end{defn}

\begin{remark}
    In short, we see that $IGED$ is a subset of $IGSD$. More precisely, 
    \begin{align*}
        IGED=IGSD\cap \{\gamma\in\otimes^2 T^*M: d\phi=0 \text{ 
 where $\phi$ is the function of $\gamma$ given by equation (\ref{c})}\}.
    \end{align*}
\end{remark}

\section{Deformations of Bismut-flat structure }

In this section, we aim to discuss essential infinitesimal generalized solitonic deformations on a Bismut-flat manifold.

\subsection{Equivalent conditions}

\begin{proposition}\label{P7}
    Let $\mathcal{G}$ be a Bismut-flat metric. The following statements are equivalent.
    \begin{enumerate}[label=(\alph*)]
        \item\label{a} $\gamma$ is an essential infinitesimal generalized solitonic deformation of $\mathcal{G}$.
        \item $\overline{\nabla}\gamma=0$ where $\overline{\nabla}$ is given in (\ref{MC}).
        \item\label{b} $\gamma=h-K$ satisfies $ \nabla_mh_{ij}=-\frac{1}{2}(H_{mik}K_{jk}+H_{mjk}K_{ik}),\quad\nabla_mK_{ij}=-\frac{1}{2}(H_{mjk}h_{ik}-H_{mik}h_{jk})$.
        \item $\frac{\partial^2}{\partial t^2}\lambda(\gamma)=0.$
     \end{enumerate}

\end{proposition}

\begin{proof}
$(b)\Longleftrightarrow (d)$: We recall that the second variation of generalized Einstein--Hilbert functional is given by 
\begin{align*}
    \frac{\partial^2}{\partial t^2}\lambda(\gamma)=\frac{1}{2}\int_M \langle \gamma, \overline{\triangle}\gamma \rangle dV_g=\int_M -\frac{1}{2}|\overline{\nabla}\gamma|^2  dV_g.
\end{align*}
Thus (b) and (d) are equivalent. 
\\$(b)\Longrightarrow (c)$: From the definition,
\begin{align*}
\overline{\nabla}_m\gamma_{ij}=0\Longrightarrow \nabla_m\gamma_{ij}=-\frac{1}{2}H_{mik}\gamma_{kj}+\frac{1}{2}H_{mjl}\gamma_{il}.
\end{align*}
So 
\begin{align*}
    \nabla_m h_{ij}&=\frac{1}{2}(\nabla_m\gamma_{ij}+\nabla_m\gamma_{ji})
    \\&=-\frac{1}{4}(H_{mik}\gamma_{kj}-H_{mjl}\gamma_{il}+H_{mjk}\gamma_{ki}-H_{mil}\gamma_{jl})
    \\&=-\frac{1}{2}(H_{mik}K_{jk}+H_{mjl}K_{il}),
\end{align*}
    \begin{align*}
    \nabla_m K_{ij}&=\frac{1}{2}(\nabla_m\gamma_{ji}-\nabla_m\gamma_{ij})
    \\&=-\frac{1}{4}(-H_{mik}\gamma_{kj}+H_{mjl}\gamma_{il}+H_{mjk}\gamma_{ki}-H_{mil}\gamma_{jl})
    \\&=-\frac{1}{2}(-H_{mik}h_{jk}+H_{mjl}h_{il}).
\end{align*}
\\$(c)\Longrightarrow (a)$: By direct computation, it is clear to see that $\gamma$ is non-trivial and $\mathcal{B}(\gamma)=-\frac{1}{2}\overline{\triangle}\gamma=0$.
\\$(a)\Longrightarrow (b)$: By using integration by parts, 
\begin{align*}
    \overline{\triangle}\gamma=0\Longrightarrow  \overline{\nabla}\gamma=0.
\end{align*}
\end{proof}

\begin{remark}\label{R6}
Naturally, \Cref{P7} part (c) implies that
\begin{align}
     \nonumber&\triangle h_{ij}+2\mathring{R}(h)_{ij}+(\Rc\circ h+h\circ \Rc)_{ij}=0,
     \\& \triangle K_{ij}+2\mathring{R}(K)_{ij}+(\Rc\circ K+K\circ \Rc)_{ij}=0. \label{BG2}
\end{align}
However,
the second order equations (\ref{BG2}) are not sufficient to say that $\gamma=h-K$ is an essential infinitesimal generalized solitonic deformation. We will provide some reasons in \Cref{RR}.
\end{remark}

\begin{remark}
 Recall in \Cref{C4}, we showed that every compact, complete, simply-connected Bismut flat manifold $(M,\mathcal{G})$ exists a non-trivial variation in the kernel of the second variation of $\lambda$. By \Cref{P7}, it is equivalent to say that there exists essential infinitesimal generalized solitonic deformations on $\mathcal{G}$ and its dimension is no less than $n^2$.
\end{remark}

In the following, we focus on Bismut-flat, Einstein manifolds with Einstein constant $\mu$. \Cref{P7} suggests us the following proposition.
\begin{proposition}\label{P8}
  Let $(M,g,H)$ be a Bismut-flat, Einstein manifold with Einstein constant $\mu$. Suppose $4\mu$ is not an eigenvalue of Laplacian $\triangle$, then on $(M,g,H)$
  \begin{align*}
      IGED=IGSD,
  \end{align*}
  i.e., all essential infinitesimal generalized solitonic deformations are essential infinitesimal generalized Einstein deformations.
\end{proposition}

\begin{proof}
Given $\gamma\in IGSD$, (\ref{BG2}) implies that 
\begin{align*}
    \triangle h_{ij}+2\mathring{R}(h)_{ij}+(\Rc\circ h+h\circ \Rc)_{ij}=0.
\end{align*}
Take the trace, we get
 \begin{align*}
    \triangle \tr_gh+4\mu\tr_gh=0.
\end{align*}
Thus, $\tr_gh=0$ and $\gamma\in IGED$.
\end{proof}

\subsection{Existence of infinitesimal generalized Einstein deformation}

In the following, we would like to discuss the existence of infinitesimal generalized Einstein deformations on the Bismut-flat case. Before we mention the main result, we have two general lemmas.

\begin{lemma}\label{L10}
Suppose $(M^n,g,H)$ is a compact Bismut-flat manifold. If $\gamma=h-K$ is an essential infinitesimal generalized solitonic deformation on $(M^n,g,H)$, then 
\begin{align*}
    \big(\mathring{R}(h),h\big)_{L^2}=-\frac{1}{2}\|\divg h\|_{L^2}^2=-\big(\mathring{R}(K),K\big)_{L^2}.
\end{align*}
\end{lemma}
\begin{proof}
Define
\begin{align*}
    (Dh)_{ijk}\coloneqq \nabla_i h_{jk}+\nabla_j h_{ki}+\nabla_k h_{ij}.
\end{align*}
By using \Cref{P7} part (c), we compute
\begin{align*}
   (Dh)_{mij}&=\nabla_m h_{ij}+\nabla_i h_{jm}+\nabla_j h_{mi}
   \\&=\frac{1}{2}[H_{mik}K_{kj}+H_{mjk}K_{ki}+H_{ijk}K_{km}+H_{imk}K_{kj}+H_{jmk}K_{ki}+H_{jik}K_{km}]
   \\&=0.
\end{align*}
On the other hand, 
\begin{align*}
    \|Dh\|_{L^2}^2&=\int_M (\nabla_ih_{jk}+\nabla_jh_{ki}+\nabla_kh_{ij})^2 dV_g=3\|\nabla h\|_{L^2}^2+6\int_M \nabla_ih_{jk}\nabla_jh_{ki} dV_g.
\end{align*}
Note that 
\begin{align*}
 \int_M \nabla_ih_{jk}\nabla_jh_{ki} dV_g&=-\int_M h_{jk}\nabla_i\nabla_jh_{ki} dV_g
 \\&=-\int_M h_{jk}(\nabla_j\nabla_ih_{ki}-\mathring{R}(h)_{jk}+R_{jl}h_{lk})dV_g
 \\&=\int_M |\text{div}h|^2+\langle \mathring{R}h,h\rangle- R_{jl}h_{jk}h_{lk}dV_g
\end{align*}
so we conclude that 
\begin{align*}
    \|Dh\|_{L^2}^2&=3\|\nabla h\|_{L^2}^2+6\|\text{div} h\|_{L^2}^2+6(\mathring{R}h,h)_{L^2}-6\int_M  R_{jl}h_{jk}h_{lk}dV_g
    \\&=6\|\text{div} h\|_{L^2}^2+12(\mathring{R}h,h)_{L^2},
\end{align*}
where we use the fact that $\triangle h+2\mathring{R}(h)+(\Rc\circ h+h\circ\Rc)=0$. Next, we compute
\begin{align*}
    \langle\mathring{R}(K),K \rangle=R_{iabj}K_{ij}K_{ab}=\frac{1}{4}K_{ij}K_{ab}H_{ijl}H_{abl}-\frac{1}{4}K_{ij}K_{ab}H_{ibl}H_{ajl}=|\divg h|^2-\frac{1}{4}K_{ij}K_{ab}H_{ial}H_{jbl}.
\end{align*}
Since
\begin{align*}
   R_{iabj}K_{ij}K_{ab}=-\frac{1}{2}R_{ijab}K_{ij}K_{ab}=\frac{1}{4}K_{ij}K_{ab}H_{ial}H_{jbl},
\end{align*}
we have
\begin{align*}
    |\divg h|^2=\frac{1}{2}K_{ij}K_{ab}H_{ial}H_{jbl}=2\langle\mathring{R}(K),K \rangle.
\end{align*}

\end{proof}

\begin{lemma}\label{L11}
Suppose $(M^n,g,H)$ is a compact Bismut-flat manifold. If $\gamma=h-K$ is an essential infinitesimal generalized solitonic deformation on $(M^n,g,H)$, then 
\begin{align*}
     \|\nabla h\|_{L^2}= \|\nabla K\|_{L^2}, \quad \int_M R_{ij}h_{jk}h_{ik} dV_g-\|\divg h\|^2_{L^2}= \int_M R_{ij}K_{jk}K_{ik} dV_g.
\end{align*}

\end{lemma}
\begin{proof}
By definition,
\begin{align*}
    |\nabla K|^2&=\nabla_m K_{ij}\cdot  \nabla_m K_{ij}
    \\&=\frac{1}{4}(H_{mik}h_{kj}-H_{m jk}h_{ki})(H_{mil}h_{lj}-H_{mjl}h_{li})
    \\&=\frac{1}{2}H_{k l}^2 h_{kj}h_{lj}-\frac{1}{2}H_{mik}H_{mjl}h_{kj}h_{li}
    \\&=2R_{kl}h_{kj}h_{lj}+2\langle\mathring{R}(h),h \rangle.
\end{align*}
Therefore,
\begin{align*}
   \|\nabla h\|_{L^2}^2= \|\nabla K\|_{L^2}^2.
\end{align*}
It also implies
\begin{align*}
    \int_M R_{ij}h_{jk}h_{ik}dV_g&=\int_M \frac{1}{2}|\nabla h|^2-\langle\mathring{R}(h),h \rangle dV_g
    \\&=\int_M \frac{1}{2}|\nabla K|^2+\langle\mathring{R}(K),K \rangle dV_g
    \\&=\int_M |\divg h|^2+R_{ij}K_{jk}K_{ik} dV_g.
\end{align*}

\end{proof}

\begin{theorem}\label{T8}
Suppose $(M^n,g,H)$ is a compact Bismut-flat manifold with positive sectional curvature. Then, there does not exist any essential infinitesimal generalized Einstein deformation.
\end{theorem}
\begin{proof}
Suppose $\gamma=h-K$ is an essential infinitesimal general Einstein deformation. Taking the trace of \Cref{P7} part (c), we get
\begin{align*}
    \nabla_m(\tr_gh)=-2(\divg h)_m.
\end{align*}
From the definition, we have $\divg h=0$. Then, \Cref{L10} and \Cref{L11} imply   
\begin{align*}
    \big(\mathring{R}(h),h\big)_{L^2}=\big(\mathring{R}(K),K\big)_{L^2}=0,
\end{align*}
and
\begin{align*}
     \int_M R_{ij}h_{jk}h_{ik} dV_g= \int_M R_{ij}K_{jk}K_{ik} dV_g=\frac{1}{2}\|\nabla h\|_{L^2}^2=\frac{1}{2}\|\nabla K\|_{L^2}^2.
\end{align*}
Note that 
\begin{align*}
    \langle \mathring{R}(K),K \rangle=\sum_{i,j}sec(e_i,e_j)(K_{ij})^2=0.
\end{align*}
We can conclude that $K=h=0$ since Ricci curvatures are positive.
\end{proof}

\begin{remark}
Suppose $(M^n,g,H)$ is a compact Bismut--flat, Ricci flat manifold. If $\gamma=h-K$ is an essential infinitesimal generalized Einstein deformation, then the proof in \Cref{T8} reduces to say that
\begin{align*}
    \nabla h=0 \quad \text{and} \quad \tr_gh=0.
\end{align*}
In fact, $(M^n,g,H)$ is a flat manifold ($H=0$) and we observe that $h$ is also an essential infinitesimal Einstein deformation that matches our expectation. 

\end{remark}

\begin{example}
 Suppose $(M,g,H)$ is a compact semisimple Lie group with bi-invariant metric $g$ and 3-form $H$ is constructed by $g^{-1}H(X,Y)=[X,Y]$ where $X,Y$ are left-invariant vector fields. In fact, the Killing form is negative definite thus $(M,g,H)$ is a compact, Bismut--flat, Einstein manifold with Einstein constant $\mu$. Recall that in the proof of \Cref{C4} we constructed essential infinitesimal generalized solitonic deformations which are defined by
 \begin{align*}
     \gamma_{ij}=\alpha_i\otimes \beta_j
 \end{align*}
where $\alpha$ is a left-invariant 1-form and $\beta$ is a right-invariant 1-form. Since $\gamma$ is $\overline{\nabla}$-parallel, we compute that
\begin{align*}
    \triangle \gamma_{ij}&=\nabla_m(\frac{1}{2}H_{mjk}\gamma_{ik}-\frac{1}{2}H_{mik}\gamma_{kj})
    \\&=\frac{1}{2}H_{mjk}\nabla_m\gamma_{ik}-\frac{1}{2}H_{mik}\nabla_m\gamma_{kj}
    \\&=-2\mu\gamma_{ij}-\frac{1}{2}H_{mjk}H_{mil}\gamma_{lk}.
\end{align*}
In other words, $\triangle\gamma+2\mathring{R}(\gamma)+2\mu\gamma=0$. By taking the trace, we see that 
\begin{align*}
    \triangle(\tr_g\gamma)+4\mu(\tr_g\gamma)=0.
\end{align*}
In this example, it is clear to see that $\tr_g\gamma$ can not be 0 i.e., $\gamma$ must be an essential infinitesimal generalized solitonic deformation and $\tr_g\gamma$ is an eigenfunction of $\triangle$ with eigenvalue $4\mu$. In $S^3$ case, we will see that all infinitesimal solitonic deformations are constructed by using left and right invariant one forms. The arguments above shows that $IGED=\emptyset$ which matches the result of  \Cref{T8}.

\end{example}

\subsection{Dimension 3 Case.}

In this section, we focus on the 3-dimensional case. Recall that in \cite{K} Corollary 2.22, we show that any compact, 3-dimensional generalized Einstein manifold has constant nonnegative sectional curvatures. In this section, we assume  $(M,g,H)$ is a 3-dimensional,  compact, Bismut-flat manifold with positive Einstein constant $\mu$ and reader should note that it must be a quotient of $S^3$.

\begin{proposition}\label{P9}
Suppose $(M,g,H)$ is a 3-dimensional, compact, Bismut-flat manifold with positive Einstein constant $\mu$. Then, any essential infinitesimal generalized solitonic deformation is of the form
\begin{align*}
    \gamma=2\mu ug+\nabla^2 u-\frac{1}{2}d^*(uH)
\end{align*}
where $u$ is an eigenfunction with eigenvalue $4\mu$.
\end{proposition}
\begin{proof}
We claim that
\begin{align*}
    IGSD=IGED\oplus \{\gamma=h-K: h=2\mu ug+\nabla^2u,\kern0.5em K=\frac{1}{2}d^*(uH) \text{ where } \triangle u+4\mu u=0\}.
\end{align*}
For any $\gamma=h-K\in IGSD$, the proof in \Cref{P8} suggests that $\tr_gh$ satisfies $\triangle\tr_gh+4\mu \tr_gh=0$. We take $u=\frac{1}{2\mu}\tr_gh$ $(\mu>0 $ in this case) and let
\begin{align*}
    \Tilde{\gamma}=\gamma-(2\mu ug+\nabla^2u-\frac{1}{2}d^*(uH)).
\end{align*}
In \cite{K} Remark 4.18, we see that $\overline{\triangle}(2\mu ug+\nabla^2u-\frac{1}{2}d^*(uH))=0$ so 
$\overline{\triangle}\Tilde{\gamma}=0$. \Cref{P7} implies that $\Tilde{\gamma}\in IGED$ and we finish the proof of claim. Based on this claim, the proposition follows from \Cref{T8}. 
\end{proof}

\begin{remark}\label{RR}
Let $(M^n,g,H)$ be a Bismut-flat, Einstein manifold with Einstein constant $\mu$ and $n\geq 4$. Suppose $u$ is an eigenfunction with eigenvalue $4\mu$. By direct computation, we can pick nonzero constants $a,c$ such that 
\begin{align*}
    \gamma=a\mu ug+c\nabla^2u-d^*(uH)
\end{align*}
satisfy (\ref{BG2}) and $\gamma$ is non-trivial. Since $\lambda$ is diffeomorphism invariant, the variation of the form $ L_{X}g-i_XH $ is trivial for any vector field $X$. In particular, we pick $X=-\frac{c}{2}\nabla u$. Then, 
\begin{align*}
    \frac{d^2}{dt^2}\lambda(a\mu ug, (1+\frac{c}{2})d^*(uH) )= \frac{d^2}{dt^2}\lambda (a\mu ug+c\nabla^2u,d^*(uH))+\frac{d^2}{dt^2}\lambda (-c\nabla^2 u, \frac{c}{2}d^*(uH)).
\end{align*}
If $ \gamma\in IGSD$, then there exists a conformal variation in the kernel of second variation. This result contradicts with Lemma 4.9 in \cite{K}. Therefore, the argument in the proof of \Cref{P9} only works in 3-dimensional case.
\end{remark}

Recall that in the unit sphere $S^3$ with Einstein constant $\mu=2$, the eigenvalues of Laplacian are $k(k+2)$ where $k=1,2,...$ and their corresponding eigenfunctions are homogenous harmonic polynomials of degree $k$ in $\mathbb{R}^4$. Therefore, we get the following corollary.
\begin{corollary}\label{C10}
  Let $\{x_1,x_2,x_3,x_4\}$ be a coordinate of $\mathbb{R}^4$. Any essential infinitesimal generalized solitonic deformation $\gamma$ on the unit sphere $S^3$ is of the form
  \begin{align*}
      \gamma= 4 ug+\nabla^2 u-\frac{1}{2}d^*(uH)
  \end{align*}
  where $u$ is spanned by functions 
  \begin{align*}
      \{x_1x_2,x_1x_3,x_1x_4,x_2x_3,x_2x_4,x_3x_4,x_1^2-x_2^2,x_1^2-x_3^2,x_1^2-x_4^2\}
  \end{align*}
 on the unit sphere $|x|=1$. The dimension of essential infinitesimal generalized solitonic deformations on $S^3$ is 9.   
\end{corollary}

\subsection{Second order Integrability}
In the last part of this section, we would like to see that not all of these essential infinitesimal generalized solitonic deformations are integrable up to second order. Let $(M,g,H)$ be a 3-dimensional, compact, Bismut-flat manifold with positive Einstein constant $\mu$. Due to \Cref{L9},  $\gamma=2\mu ug+\nabla^2u-\frac{1}{2}d^*(uH)\in IGSD$ is integrable up to second order if and only if  
\begin{align*}
    \int_M \Big\langle \mathcal{R}''(2\mu ug+\nabla^2u-\frac{1}{2}d^*(uH),2\mu ug+\nabla^2u-\frac{1}{2}d^*(uH)),2\mu wg+\nabla^2w-\frac{1}{2}d^*(wH) \Big \rangle dV_g=0.
\end{align*}
for any $w$ satisfying $\triangle w+4\mu w=0$. Since $\lambda$ is diffeomorphism invariant, we can normalize and consider  
\begin{align*}
    \int_M \Big\langle \mathcal{R}''( ug-\frac{1}{2\mu}d^*(uH),ug-\frac{1}{2\mu}d^*(uH)),wg-\frac{1}{2\mu}d^*(wH) \Big \rangle dV_g.
\end{align*}

In the following, we fix a background 3-form $H$ and consider a family of smooth metrics and 2-forms $(g_t,b_t)$ with 
\begin{align*}
        g_t=g+tug,\quad b_t=0+t \frac{d^*(uH)}{2\mu}.
\end{align*}
Then,
\begin{align*}
    H_t=H+db_t=H+t\frac{dd^*(uH)}{2\mu}=(1-\frac{\triangle u }{2\mu} t)H.
\end{align*}
In addition, we take $f_t$ to be the minimizer of $\lambda(g_t,b_t)$. For convenience, we denote 
\begin{align*}
    &\frac{\partial}{\partial t}g_t=ug=\frac{u}{1+tu} g_t\coloneqq u_t g_t, 
    \\&\frac{\partial}{\partial t}H_t=\frac{-\triangle u}{2\mu} H=-\frac{\frac{\triangle u}{2\mu}}{1-\frac{\triangle u t}{2\mu}}H_t\coloneqq \Tilde{u}_t H_t,
\end{align*}
i.e., $u_t=\frac{u}{1+tu}$ and $\Tilde{u}_t=-\frac{\triangle u}{2\mu-(\triangle u) t}$. Therefore,
\begin{align*}
    u_0=u,\quad\Tilde{u}_0=2u,\quad\frac{\partial}{\partial t}\Big|_{t=0} u_t=-u^2, \quad \frac{\partial}{\partial t}\Big|_{t=0} \Tilde{u}_t=-(\frac{\triangle u}{2\mu})^2=-4u^2.
\end{align*}
By using above notation, we can derive the evolution formulas which we record in appendix B. We then derive the following theorem.

\begin{theorem}\label{T9}
 Suppose $(M,g,H)$ is a 3-dimensional, compact, Bismut-flat manifold with positive Einstein constant $\mu$. Then,
 \begin{align*}
    \int_M \Big\langle \mathcal{R}''( ug-\frac{1}{2\mu}d^*(uH),ug-\frac{1}{2\mu}d^*(uH)),wg-\frac{1}{2\mu}d^*(wH) \Big \rangle dV_g=-6\int_M \mu u^2w dV_g,
\end{align*}
where $u,w$ are eigenfunctions of Laplacian with eigenvalue $4\mu$. Therefore, $\gamma=2\mu ug+\nabla^2u-\frac{1}{2}d^*(uH)\in IGSD$ is integrable up to second order if and only if 
\begin{align*}
    \int_M u^2w=0
\end{align*}
for any $w$ satisfying $\triangle w+4\mu w=0$.
\end{theorem}

\begin{proof}
Using the above notation, we see that 
\begin{align*}
 \int_M \Big\langle \mathcal{R}''( ug-\frac{1}{2\mu}d^*(uH),ug-\frac{1}{2\mu}d^*(uH)),wg-\frac{1}{2\mu}d^*(wH) \Big \rangle dV_g=\int_M \Big\langle \frac{\partial^2 }{\partial t^2}\Big|_{t=0}\Rc^{H,f}_t, wg-\frac{1}{2\mu}d^*(wH)  \Big\rangle dV_g.
\end{align*}
By \Cref{AL3}, we have 
\begin{align*}
   \frac{\partial^2 }{\partial t^2}\Big|_{t=0}\Rc^{H,f}_t&=\Big( (|\nabla u|^2-2u^2\mu)g+\frac{1}{2}\nabla u\otimes \nabla u+u\nabla^2u+(\nabla^2f'')\Big|_{t=0}\Big)
  \\&\kern2em -\Big(\frac{5}{2}u(i_{\nabla u}H)+\frac{1}{2}(i_{\nabla f''}H)\Big|_{t=0}\Big).
\end{align*}
Then,
\begin{align*}
    &\Big\langle (|\nabla u|^2-2u^2\mu)g+\frac{1}{2}\nabla u\otimes \nabla u+u\nabla^2u+(\nabla^2f'')\Big|_{t=0},wg \Big\rangle
    \\&\kern1em=3w (|\nabla u|^2-2u^2\mu)+\frac{w}{2}|\nabla u|^2+u w\triangle u+w\triangle f''|_{t=0}
    \\&\kern1em=\frac{7}{2}w|\nabla u|^2-10u^2w\mu+w(7u^2\mu-\frac{7}{4}|\nabla u|^2)
    \\&\kern1em=\frac{7}{4}|\nabla u|^2-3u^2w\mu.
\end{align*}
Also,
\begin{align*}
   &\Big\langle (\frac{5}{2}u(i_{\nabla u}H)+\frac{1}{2}(i_{\nabla f''}H),\frac{1}{2\mu}d^*(wH)\Big\rangle
   \\&\kern1em =\Big(\frac{5}{2}u\nabla_luH_{lij}+\frac{1}{2}\nabla_l(f''|_{t=0})H_{lij}\Big)(\frac{-1}{2\mu}\nabla_kwH_{kij})
   \\&\kern1em=-5u\langle \nabla u,\nabla w\rangle-\langle \nabla f''|_{t=0},\nabla w\rangle.
\end{align*}
Therefore, 
\begin{align*}
    &\int_M \Big\langle \mathcal{R}''( ug-\frac{1}{2\mu}d^*(uH),ug-\frac{1}{2\mu}d^*(uH)),wg-\frac{1}{2\mu}d^*(wH) \Big \rangle dV_g
    \\&\kern1em =\int_M \frac{7}{4}|\nabla u|^2-3u^2w\mu-5u\langle \nabla u,\nabla w\rangle-\langle \nabla f''|_{t=0},\nabla w\rangle dV_g
    \\&\kern1em =\int_M 4u^2w\mu-5u\langle \nabla u,\nabla w\rangle dV_g
    \\&\kern1em=-6\int_M u^2w\mu dV_g,
\end{align*}
where we use integration by parts to see that 
\begin{align*}
    \int_M u\langle \nabla u,\nabla w\rangle dV_g&=-\int_M u^2\triangle w dV_g-\int_M  u\langle \nabla u,\nabla w\rangle dV_g
    \\&=4\int_Mu^2w\mu dV_g-\int_M u\langle \nabla u,\nabla w\rangle dV_g.
\end{align*}
\end{proof}

By this theorem, we see that not all $\gamma\in IGSD$ are integrable up to second order. For example,
\begin{align*}
    u=x_1^2-x_2^2, \quad w=x_1^2-x_3^2
\end{align*}
then 
\begin{align*}
    \int_{S^3} (x_1^2-x_2^2)^2(x_1^2-x_3^2)dV\neq 0.
\end{align*}
But, there are some examples that satisfy the criterion. For instance,
\begin{align*}
    u=x_1x_2\pm x_3x_4.
\end{align*}
Then,
\begin{align*}
    \int_{S^3}(x_1x_2\pm x_3x_4)^2w dV=0 \quad \text{for all eigenfunctions $w$. }
\end{align*}
\appendix

\section{Variation formulas}
In this appendix, let $(M,g,H)$ be a 3-dimensional, compact, Bismut-flat manifold with positive Einstein constant $\mu$ and $u$ is an eigenfunction of Laplacian with eigenvalue $4\mu$. Consider a family $(g_t,H_t,f_t)\in \Gamma(S^2M)\times \Omega^3\times C^\infty(M)$ with 
\begin{align*}
    g_t=(1+tu)g,\quad H_t=(1+2tu)H,
\end{align*}
and $f_t$ is the minimizer of $\lambda(g_t,H_t)$. We have the following results.
\begin{lemma}\label{AL1}
Given a smooth metric and a 3-form $(g,H)$, define $(g_t,H_t)\in \Gamma(S^2M)\times \Omega^3$ by
\begin{align*}
    g_t=(1+tu)g,\quad H_t=(1+2tu)H,
\end{align*}
where $u$ is an eigenfunction of Laplacian with eigenvalue $4\mu$. Then,
\begin{align*}
     \nonumber \frac{\partial }{\partial t}\Gamma_{ij}^k &=\frac{1}{2}\Big(\nabla_iu_t(g_t)_j^k+\nabla_ju_t(g_t)_i^k-\nabla^ku_t(g_t)_{ij}\Big),
      \\  \nonumber \frac{\partial }{\partial t}R_{ij}&=-\frac{1}{2}(\triangle u_t) (g_t)_{ij}-\frac{1}{2}\nabla_i\nabla_ju_t    
      \\ \nonumber  \frac{\partial }{\partial t} R&=-2\triangle u_t-u_tR_t,
    \\ \nonumber \frac{\partial }{\partial t} \nabla_i\nabla_j&=-\frac{1}{2}\Big(\nabla_iu_t\nabla_j+\nabla_ju_t\nabla_i-\nabla^ku_t(g_t)_{ij}\nabla_k\Big),
    \\ \nonumber \frac{\partial }{\partial t} \triangle&=-u_t\triangle+\frac{1}{2}\nabla^ku_t\nabla_k,
    \\  \frac{\partial }{\partial t} H^2_{ij}&=(-2u_t+2\Tilde{u}_t)(H_t^2)_{ij},
    \\  \nonumber \frac{\partial }{\partial t}|H|^2&=(-3u_t+2\Tilde{u}_t)|H_t|^2,
      \\  \nonumber\frac{\partial }{\partial t}(d^*H)_{ij}&=-u_t(d^*H_t)_{ij}+d^*(\Tilde{u_t}H_t)_{ij}+\frac{3}{2}(i_{\nabla u_t}H)_{ij}.
\end{align*}
where $u_t=\frac{u}{1+tu}$, $\Tilde{u}_t=-\frac{\triangle u}{2\mu-(\triangle u) t}=\frac{2u}{1+2tu}$. If $f(t)$ is a family of smooth functions, we have
\begin{align*}
    \frac{\partial }{\partial t}|\nabla f|^2&=-u_t|\nabla f_t|^2+2\langle \nabla f_t,\nabla f_t' \rangle,
    \\\frac{\partial }{\partial t}(i_{\nabla f}H)&=(-u_t+\Tilde{u}_t)(i_{\nabla {f_t}}H)+(i_{\nabla {f'_t}}H).
\end{align*}
\end{lemma}
\begin{proof}
This lemma is followed by direct computation. One can consult with \cite{GRF} chapter 5.1.
\end{proof}

Besides, we know that a Bismut--flat manifold is a critical point of $\lambda$ so we are able to compute the derivative of $f$.
\begin{lemma}\label{AL2}
Let $(M,g,H)$ be a 3-dimensional, compact, Bismut-flat manifold with positive Einstein constant $\mu$ and $u$ is an eigenfunction of Laplacian with eigenvalue $4\mu$. Suppose $(g_t,H_t,f_t)\in \Gamma(S^2M)\times \Omega^3\times C^\infty(M)$ with 
\begin{align*}
    g_t=(1+tu)g,\quad H_t=(1+2tu)H,
\end{align*}
and $f_t$ is the minimizer of $\lambda(g_t,H_t)$. Then,
\begin{align*}
   \frac{\partial }{\partial t}\Big|_{t=0}f=\frac{u}{2}\quad \text{and}\quad \triangle\Big(\frac{\partial^2 }{\partial t^2}\Big|_{t=0}f\Big)=7\mu u^2-\frac{7}{4}|\nabla u|^2.
\end{align*}
\end{lemma}
\begin{proof}
We recall that the derivative of $f$ satisfies the equation (\ref{Df}) so in this case
\begin{align*}
    0=2\triangle (\frac{\partial }{\partial t}\Big|_{t=0}f)-\triangle u,\quad \int_M 3u-2(\frac{\partial }{\partial t}\Big|_{t=0}f) dV_g=0.
\end{align*}
Since $\int_M u dV_g=0$, we can conclude that $\frac{\partial }{\partial t}\Big|_{t=0}f=\frac{u}{2}$. For the second derivative, we observe that
\begin{align*}
    0= \frac{\partial^2 }{\partial t^2}\Big|_{t=0}\lambda= \frac{\partial^2 }{\partial t^2}\Big|_{t=0}(R-\frac{1}{12}|H|^2+2\triangle f-|\nabla f|^2).
\end{align*}
We compute
\begin{align*}
     \frac{\partial^2 }{\partial t^2}\Big|_{t=0}R&=\frac{\partial }{\partial t}\Big|_{t=0}(-2\triangle u_t-u_t R_t)
     \\&=\Big(-2\triangle'u_t-2\triangle u_t'-u_t'R_t-u_tR_t'\Big)\Big|_{t=0}
     \\&=(2u\triangle u-|\nabla u|^2)+2\triangle(u^2)+u^2R-u(-2\triangle u-uR)
     \\&=8u\triangle u+3|\nabla u|^2+2u^2R
     \\&=-26u^2\mu+3|\nabla u|^2,
\end{align*}
\begin{align*}
    \frac{\partial^2 }{\partial t^2}\Big|_{t=0}|H|^2&=\frac{\partial }{\partial t}\Big|_{t=0}\Big((-3u_t+2\Tilde{u}_t)|H_t|^2\Big)
 \\&=\Big((-3u_t+2\Tilde{u}_t)'|H_t|^2+(-3u_t+2\Tilde{u}_t)(|H_t|^2)'\Big)\Big|_{t=0}
 \\&=-5u^2|H|^2+(-3u_t+2\Tilde{u}_t)^2|H_t|^2\Big|_{t=0}
 \\&=-4u^2|H|^2
 \\&=-48u^2\mu,
\end{align*}
\begin{align*}
    (\triangle'f')\Big|_{t=0}&=(-u_t\triangle f'+\frac{1}{2}\langle\nabla u_t,\nabla f'\rangle)\Big|_{t=0}
    \\&=-\frac{u}{2}\triangle u+\frac{1}{4}|\nabla u|^2
    \\&=2u^2\mu+\frac{1}{4}|\nabla u|^2,
\end{align*}
\begin{align*}
     \frac{\partial^2 }{\partial t^2}\Big|_{t=0}|\nabla f|^2&=\frac{\partial }{\partial t}\Big|_{t=0}\Big(-u_t|\nabla f_t|^2+2\langle \nabla f_t,\nabla f_t' \rangle\Big)
     \\&=\Big(-u_t'|\nabla f_t|^2-u_t(|\nabla f_t|^2)'-2u_t \langle \nabla f_t,\nabla f_t' \rangle+2|\nabla f'_t|^2+2\langle \nabla f_t,\nabla f_t'' \rangle\Big)\Big|_{t=0}
     \\&=\frac{1}{2}|\nabla u|^2.
\end{align*}
Therefore,
\begin{align*}
    0&=\frac{\partial^2 }{\partial t^2}\Big|_{t=0}(R-\frac{1}{12}|H|^2+2\triangle f-|\nabla f|^2)
    \\&=\Big( R''-\frac{1}{12}(|H|^2)''+2\triangle''f+4\triangle'f'+2\triangle f''-(|\nabla f|^2)''\Big)\Big|_{t=0}
    \\&=-14u^2\mu+\frac{7}{2}|\nabla u|^2+2\triangle f''.
\end{align*}
\end{proof}

Next, we are able to compute the second derivative.
\begin{lemma}\label{AL3}
Let $(M,g,H)$ be a 3-dimensional, compact, Bismut-flat manifold with 
 positive Einstein constant $\mu$ and $u$ is an eigenfunction of Laplacian with eigenvalue $4\mu$. Suppose $(g_t,H_t,f_t)\in \Gamma(S^2M)\times \Omega^3\times C^\infty(M)$ with 
\begin{align*}
    g_t=(1+tu)g,\quad H_t=(1+2tu)H,
\end{align*}
and $f_t$ is the minimizer of $\lambda(g_t,H_t)$. We have the following results.
\begin{align*}
    \frac{\partial^2 }{\partial t^2}\Big|_{t=0}\Rc&=(\frac{1}{2}|\nabla u|^2-4u^2\mu)g+\frac{3}{2}\nabla u\otimes \nabla u+u\nabla^2u,
    \\\frac{\partial^2 }{\partial t^2}\Big|_{t=0}H^2&=-8u^2\mu g,
    \\\frac{\partial^2 }{\partial t^2}\Big|_{t=0}\nabla^2f&=-\nabla u\otimes \nabla u+\frac{1}{2}|\nabla u|^2 g+(\nabla^2f''),
    \\\frac{\partial^2 }{\partial t^2}\Big|_{t=0}(d^*H)&=4u(i_{\nabla u}H),
    \\\frac{\partial^2 }{\partial t^2}\Big|_{t=0}(i_{\nabla f}H)&=u(i_{\nabla u}H)+(i_{\nabla f''}H).
\end{align*}
Furthermore, we have
\begin{align*}
  \frac{\partial^2 }{\partial t^2}\Big|_{t=0}\Rc^{H,f}_t&=\Big( (|\nabla u|^2-2u^2\mu)g+\frac{1}{2}\nabla u\otimes \nabla u+u\nabla^2u+(\nabla^2f'')\Big|_{t=0}\Big)
  \\&\kern2em -\Big(\frac{5}{2}u(i_{\nabla u}H)+\frac{1}{2}(i_{\nabla f''}H)\Big).
\end{align*}

\end{lemma}
\begin{proof}
Following \Cref{AL1}, we compute
\begin{align*}
    \frac{\partial^2 }{\partial t^2}\Big|_{t=0}R_{ij}&=\frac{\partial }{\partial t}\Big|_{t=0}\Big(-\frac{1}{2}\triangle u_t(g_t)_{ij}-\frac{1}{2}\nabla_i\nabla_ju_t\Big)
    \\&=\Big(-\frac{1}{2}\triangle'u_t(g_t)_{ij}-\frac{1}{2}\triangle u_t'(g_t)_{ij}-\frac{1}{2}\triangle u_t(g_t)'_{ij}-\frac{1}{2}(\nabla_i\nabla_j)'u_t--\frac{1}{2}\nabla_i\nabla_ju'_t\Big)\Big|_{t=0}
    \\&=(\frac{u}{2}\triangle u-\frac{1}{4}|\nabla u|^2)g_{ij}+\frac{1}{2}\triangle(u)^2g_{ij}-\frac{1}{2}u\triangle u g_{ij}+\frac{1}{2}\nabla_iu\nabla_ju-\frac{1}{4}|\nabla u|^2g_{ij}+\frac{1}{2}\nabla_i\nabla_j(u)^2
    \\&=u\triangle ug_{ij}+\frac{1}{2}|\nabla u|^2g_{ij}+\frac{3}{2}\nabla_iu\nabla_ju+u\nabla_i\nabla_ju
    \\&=(\frac{1}{2}|\nabla u|^2-4u^2\mu)g_{ij}+\frac{3}{2}\nabla_iu\nabla_ju+u\nabla_i\nabla_ju,
\end{align*}
\begin{align*}
     \frac{\partial^2 }{\partial t^2}\Big|_{t=0}H^2_{ij}&=\frac{\partial }{\partial t}\Big|_{t=0}\Big((-2u_t+2\Tilde{u}_t)(H^2_t)_{ij}\Big)
     \\&=\Big((-2u'_t+2\Tilde{u}'_t)(H^2_t)_{ij}+(-2u_t+2\Tilde{u}_t)(H^2_t)'_{ij}\Big)\Big|_{t=0}
     \\&=-6u^2H^2_{ij}+(-2u_t+2\Tilde{u}_t)^2(H^2_t)_{ij}\Big|_{t=0}
     \\&=-2u^2H^2_{ij}
     \\&=-8u^2\mu g_{ij},
\end{align*}
\begin{align*}
     \frac{\partial^2 }{\partial t^2}\Big|_{t=0}\nabla_i\nabla_jf&=\Big(2(\nabla_i\nabla_j)'f'+\nabla_i\nabla_jf''\Big)\Big|_{t=0}
     \\&=\Big(-\nabla_iu_t\nabla_jf'-\nabla_ju_t\nabla_if'+\langle \nabla u,\nabla f'\rangle g_{ij}+\nabla_i\nabla_jf''\Big)\Big|_{t=0}
     \\&=-\nabla_iu\nabla_ju+\frac{1}{2}|\nabla u|^2 g_{ij}+(\nabla^2f'')\Big|_{t=0},
\end{align*}
\begin{align*}
   \frac{\partial^2 }{\partial t^2}\Big|_{t=0}(d^*H)&=\frac{\partial }{\partial t}\Big|_{t=0}\Big(-u_t(d^*H_t)+d^*(\Tilde{u_t}H_t)+\frac{3}{2}(i_{\nabla u_t}H)\Big)
   \\&=\Big( -u_t'(d^*H_t)-u_t(d^*H_t)'+ d^*(\Tilde{u_t}H_t)'+\frac{3}{2}(i_{\nabla u_t}H)'\Big)\Big|_{t=0}
   \\&=\Big(-u_t(-u_t(d^*H_t)+d^*(\Tilde{u_t}H_t)+\frac{3}{2}(i_{\nabla u_t}H))+(-u_t d^*(\Tilde{u}_tH_t)+d^*(\Tilde{u_t}H_t)'+\frac{3}{2}\Tilde{u}_t(i_{\nabla u_t}H))
   \\&\kern2em+\frac{3}{2}(-u_t+\Tilde{u}_t)(i_{\nabla {u_t}}H)+\frac{3}{2}(i_{\nabla {u'_t}}H)\Big)\Big|_{t=0}
   \\&=-4ud^*(uH)
   \\&=4u(i_{\nabla u}H),
\end{align*}
\begin{align*}
     \frac{\partial^2 }{\partial t^2}\Big|_{t=0}(i_{\nabla f}H)&=\frac{\partial }{\partial t}\Big|_{t=0}\Big((-u_t+\Tilde{u}_t)(i_{\nabla {f_t}}H)+(i_{\nabla {f'_t}}H)\Big)
     \\&=\Big((-u'_t+\Tilde{u}'_t)(i_{\nabla {f_t}}H)+(-u_t+\Tilde{u}_t)(i_{\nabla {f_t}}H)'+(i_{\nabla {f'_t}}H)' \Big)\Big|_{t=0}
     \\&=\Big((-u_t+\Tilde{u}_t)^2(i_{\nabla {f_t}}H)+2(-u_t+\Tilde{u}_t)(i_{\nabla {f'_t}}H)+(i_{\nabla {f''_t}}H)\Big)\Big|_{t=0}
     \\&=u(i_{\nabla u}H)+(i_{\nabla {f''_t}}H)\Big|_{t=0}.
\end{align*}
Then, $\frac{\partial^2 }{\partial t^2}\Big|_{t=0}\Rc^{H,f}_t$ is derived from the definition.
\end{proof}

\bibliographystyle{plain}
\bibliography{Reference}

\begin{thebibliography}{10}

\bibitem{AF}
Ilka Agricola and Thomas Friedrich.
\newblock A note on flat metric connections with antisymmetric torsion.
\newblock {\em Differential Geometry and its Applications}, 28(4):480--487,
  2010.

\bibitem{B}
Arthur~L. Besse.
\newblock {\em Einstein manifolds}, volume~10 of {\em Ergebnisse der Mathematik
  und ihrer Grenzgebiete (3) [Results in Mathematics and Related Areas (3)]}.
\newblock Springer-Verlag, Berlin, 1987.

\bibitem{MR3430881}
Huai-Dong Cao and Chenxu He.
\newblock Linear stability of {P}erelman's {$\nu$}-entropy on symmetric spaces
  of compact type.
\newblock {\em J. Reine Angew. Math.}, 709:229--246, 2015.

\bibitem{C2}
Huai-Dong Cao and Meng Zhu.
\newblock On second variation of {P}erelman's {R}icci shrinker entropy.
\newblock {\em Math. Ann.}, 353(3):747--763, 2012.

\bibitem{MR2302600}
Bennett Chow, Sun-Chin Chu, David Glickenstein, Christine Guenther, James
  Isenberg, Tom Ivey, Dan Knopf, Peng Lu, Feng Luo, and Lei Ni.
\newblock {\em The {R}icci flow: techniques and applications. {P}art {I}},
  volume 135 of {\em Mathematical Surveys and Monographs}.
\newblock American Mathematical Society, Providence, RI, 2007.
\newblock Geometric aspects.

\bibitem{MR2061425}
Bennett Chow and Dan Knopf.
\newblock {\em The {R}icci flow: an introduction}, volume 110 of {\em
  Mathematical Surveys and Monographs}.
\newblock American Mathematical Society, Providence, RI, 2004.

\bibitem{MR0267604}
David~G. Ebin.
\newblock The manifold of {R}iemannian metrics.
\newblock In {\em Global {A}nalysis ({P}roc. {S}ympos. {P}ure {M}ath., {V}ol.
  {XV}, {B}erkeley, {C}alif., 1968)}, pages 11--40. Amer. Math. Soc.,
  Providence, R.I., 1970.

\bibitem{MG}
Mario Garcia-Fernandez.
\newblock Ricci flow, killing spinors, and t-duality in generalized geometry.
\newblock {\em Advances in Mathematics}, 350, 11 2016.

\bibitem{GRF}
Mario Garcia-Fernandez and Jeffrey Streets.
\newblock {\em Generalized {R}icci flow}, volume~76 of {\em University Lecture
  Series}.
\newblock American Mathematical Society, Providence, RI, [2021] \copyright
  2021.

\bibitem{Ko4}
N.~Koiso.
\newblock Einstein metrics and complex structures.
\newblock {\em Invent. Math.}, 73(1):71--106, 1983.

\bibitem{Ko1}
Norihito Koiso.
\newblock Nondeformability of {E}instein metrics.
\newblock {\em Osaka Math. J.}, 15(2):419--433, 1978.

\bibitem{Ko2}
Norihito Koiso.
\newblock Rigidity and stability of {E}instein metrics---the case of compact
  symmetric spaces.
\newblock {\em Osaka Math. J.}, 17(1):51--73, 1980.

\bibitem{Ko5}
Norihito Koiso.
\newblock Rigidity and infinitesimal deformability of {E}instein metrics.
\newblock {\em Osaka Math. J.}, 19(3):643--668, 1982.

\bibitem{Kroencke2014}
Klaus Kr{\"o}ncke.
\newblock {\em Stability of Einstein Manifolds}.
\newblock doctoral thesis, Universit{\"a}t Potsdam, 2014.

\bibitem{K1}
Klaus Kr\"{o}ncke.
\newblock On the stability of {E}instein manifolds.
\newblock {\em Ann. Global Anal. Geom.}, 47(1):81--98, 2015.

\bibitem{K3}
Klaus Kr\"{o}ncke.
\newblock Stability and instability of {R}icci solitons.
\newblock {\em Calc. Var. Partial Differential Equations}, 53(1-2):265--287,
  2015.

\bibitem{K4}
Klaus Kr\"{o}ncke.
\newblock Stability of {E}instein metrics under {R}icci flow.
\newblock {\em Comm. Anal. Geom.}, 28(2):351--394, 2020.

\bibitem{K2}
Klaus Kröncke.
\newblock Rigidity and infinitesimal deformability of ricci solitons.
\newblock {\em The Journal of Geometric Analysis}, 26(3):1795–1807, Apr 2015.

\bibitem{K}
Kuan-Hui Lee.
\newblock The stability of generalized {R}icci solitons, 2022.

\bibitem{JM}
John Milnor.
\newblock Curvatures of left invariant metrics on {L}ie groups.
\newblock {\em Advances in Math.}, 21(3):293--329, 1976.

\bibitem{P}
T.~Oliynyk, V.~Suneeta, and E.~Woolgar.
\newblock A gradient flow for worldsheet nonlinear sigma models.
\newblock {\em Nuclear Physics B}, 739(3):441–458, Apr 2006.

\bibitem{10.2307/j.ctt1bd6k0x}
Richard~S. Palais, M.~F. Atiyah, A.~Borel, E.~E. Floyd, R.~T. Seeley, W.~Shih,
  and R.~Solovay.
\newblock {\em Seminar on Atiyah-Singer Index Theorem. (AM-57)}.
\newblock Princeton University Press, 1965.

\bibitem{Pe}
Grisha Perelman.
\newblock {The Entropy formula for the Ricci flow and its geometric
  applications}.
\newblock 7 2006.

\bibitem{FA}
Fabio Podest\`a and Andrea Spiro.
\newblock On moduli spaces of {R}icci solitons.
\newblock {\em J. Geom. Anal.}, 25(2):1157--1174, 2015.

\bibitem{polchinski_1998}
Joseph Polchinski.
\newblock {\em String Theory}, volume~1 of {\em Cambridge Monographs on
  Mathematical Physics}.
\newblock Cambridge University Press, 1998.

\bibitem{Rubio_2019}
Roberto Rubio and Carl Tipler.
\newblock The lie group of automorphisms of a courant algebroid and the moduli
  space of generalized metrics.
\newblock {\em Revista Matemática Iberoamericana}, 36(2):485–536, Dec 2019.

\bibitem{J}
Jeffrey Streets.
\newblock Regularity and expanding entropy for connection {R}icci flow.
\newblock {\em J. Geom. Phys.}, 58(7):900--912, 2008.

\bibitem{STREETS2017506}
Jeffrey Streets.
\newblock Generalized geometry, {T}-duality, and renormalization group flow.
\newblock {\em Journal of Geometry and Physics}, 114:506--522, 2017.

\bibitem{classification}
Jeffrey Streets.
\newblock Classification of solitons for pluriclosed flow on complex surfaces.
\newblock {\em Mathematische Annalen}, 375, 12 2019.

\bibitem{Streets2020}
Jeffrey Streets.
\newblock {\em Pluriclosed Flow and the Geometrization of Complex Surfaces},
  pages 471--510.
\newblock Springer International Publishing, Cham, 2020.

\bibitem{J2}
Jeffrey Streets.
\newblock Scalar curvature, entropy, and generalized ricci flow, 2022.

\bibitem{8189878}
Jeffrey Streets and Gang Tian.
\newblock A parabolic flow of pluriclosed metrics.
\newblock {\em International Mathematics Research Notices},
  2010(16):3101--3133, 2010.

\end{thebibliography}
\nocite{*}

\end{document}